\documentclass[reqno,11pt]{amsart}
\usepackage{mathrsfs}
\usepackage{amsmath}
\usepackage{amsfonts}
\usepackage{amssymb}
\usepackage{xcolor}
\usepackage{graphicx}
\allowdisplaybreaks

\usepackage[colorlinks=true,citecolor=black,linkcolor=black]{hyperref}
\usepackage{enumerate}

\usepackage{epstopdf}
\usepackage{tikz}

\usepackage[margin=1in]{geometry}

\newtheorem{theorem}{Theorem}[section]
\newtheorem{lemma}[theorem]{Lemma}

\newtheorem{definition}[theorem]{Definition}

\newtheorem{proposition}[theorem]{Proposition}
\newtheorem{remark}[theorem]{Remark}

\newcommand{\norm}[1]{\left\lVert#1\right\rVert}
\newcommand{\abs}[1]{\left|#1\right|}

\newcommand{\R}{\mathbb{R}}
\newcommand{\C}{\mathsf{C}}

\numberwithin{equation}{section}

\newcommand{\na}{\nabla}

\newcommand{\wt}{\widetilde}

\newcommand{\eps}{\varepsilon}
\renewcommand{\epsilon}{\varepsilon}
\renewcommand{\phi}{\varphi}
\newcommand{\LO}[1]{L^{#1}(\Omega)}

\newcommand{\blue}{}

\renewcommand{\div}{\nabla\cdot}

\newcommand{\uinf}{u_{\infty}}

\newcommand{\Uinf}{U_{\infty}}

\newcommand{\E}{\mathcal{E}}
\newcommand{\D}{\mathcal{D}}
\newcommand{\intO}{\int_{\Omega}}

\newcommand{\bra}[1]{\left(#1\right)}
\newcommand{\avO}[1]{\left[#1\right]_{\Omega}}
\newcommand{\avo}[2]{\left[#1\right]_{\omega_{#2}}}

\newcommand{\RR}{\mathsf{R}}
\renewcommand{\S}{\mathsf{S}}
\newcommand{\V}{\mathsf{V}}
\newcommand{\Ed}{\mathsf{Edge}}
\newcommand{\G}{\mathsf{G}}
\newcommand{\W}{\mathsf{W}}

\renewcommand{\ln}{\log}

\title[Reaction-diffusion systems with degenerate reaction]{On the equilibriation of chemical reaction-diffusion systems with degenerate reactions}
\author[L. Desvillettes]{Laurent Desvillettes}
\address{Laurent Desvillettes\hfill\break
	Universit\'e de Paris and Sorbonne Universit\'e, CNRS and IUF, IMJ-PRG,
	F-75013 Paris, France}
\email{desvillettes@imj-prg.fr}

\author[K.-D. Phung]{Kim Dang Phung}
\address{Kim Dang Phung\hfill\break
	Institut Denis Poisson, Universit\'e d'Orl\'eans, Universit\'e de Tours \& CNRS UMR 7013, B\^{a}timent de Math\'ematiques, Rue de Chartres, BP. 6759, 45067 Orl\'eans, France}
\email{kim\_dang\_phung@yahoo.fr}
\author[B.Q. Tang]{Bao Quoc Tang}
\address{Bao Quoc Tang \hfill\break
	Institute of Mathematics and Scientific Computing, University of Graz,
	Heinrichstrasse 36, 8010 Graz, Austria}
\email{quoc.tang@uni-graz.at, baotangquoc@gmail.com}

\begin{document}
	\begin{abstract}
		\blue{The trend to equilibrium for reaction-diffusion systems
		modelling chemical reaction networks is investigated, in the case when
		reaction processes happen on subsets of the domain.
		We prove the convergence to equilibrium by directly showing functional inequalities in terms of entropy method.
		Our approach allows us to deal with nonlinearities of arbitrary orders, for which only global renormalised solutions are known to globally exist. For bounded solutions, we also prove the convergence to equilibrium
		when the diffusion as well as the reaction are degenerate, that is both diffusion and reaction processes only act on specific subsets of the domain.}
		
	\end{abstract}
	\maketitle
	
	\noindent{\small{\textbf{Classification AMS 2010:} 35A01, 35A09, 35K57, 35Q92.}}
	
	
	\tableofcontents
	
	\section{Introduction and main results}
	
	Convergence to equilibrium for reaction-diffusion systems modelling chemical reactions has been studied since the eighties in e.g. \cite{groger1983asymptotic,glitzky1996free,glitzky1997energetic} and has witnessed considerable progress recently, see e.g. \cite{Desvillettes2006,Fellner2016a,desvillettes2017trend,Mielke2015,Pierre2017,Haskovec2018,GS22, mielke2024convergence} \blue{and references therein}. Most of these works, if not all, assume a common condition: the diffusion and reactions in the system are non-degenerate, in the sense that all the chemical species diffuse and the diffusion and reactions take place everywhere in the spatial domain. When there is degeneracy, showing convergence to equilibrium is more challenging. The case of degenerate diffusion, i.e. one or some chemical species do not diffuse, has been considered for some special systems in \cite{desvillettes2015duality,fellner2018well,einav2020indirect}. To show the convergence to equilibrium in this situation, these works utilised the so-called \textit{indirect diffusion effect}, which, roughly speaking, means that the combination of the diffusion of some of the species and of reversible reactions leads to certain ``diffusion effect'' on non-diffusive species. Extending this theory to general systems still remains as an open problem. The case of degenerate reaction is much less studied, and, up to our knowledge, this has been considered only in the recent work \cite{desvillettes2021exponential}. In \cite{desvillettes2021exponential}, the reversible reaction $2\S_1 \leftrightarrows 2\S_2$ was investigated in the case when the reactions happen only in an open subset of the domain with positive measure. By utilising a technique stemming from controllability theory, namely log convexity, and the regularity of solutions, it was shown that solutions still converge exponentially to the chemical equilibrium. While the method therein is sophisticated, it seems difficult to generalise it to more general reaction networks. In this paper, we use a different approach based on proving directly entropy-entropy dissipation functional inequalities. Thanks to this, we can deal with a much larger class of systems, namely complex balanced systems with arbitrarily high orders of reactions. This approach is also sufficiently robust so that we can deal with various types of degeneracy, for instance when reactions happen in very rough domains, or when both reactions and diffusions are degenerate. \blue{The results in this work significantly extend the literature convergence to equilibrium for chemical reaction networks, cf. \cite{glitzky1996free,Mielke2015,desvillettes2017trend,fellner2018convergence}, to the situation with degenerate reactions. To our knowledge, this is also the first work showing convergence to equilibrium when both reactions and diffusion can be degenerate.}
	
	\subsection{Chemical reaction-diffusion systems}
	Consider $m$ chemical species $\S_1, \ldots, \S_m$ reacting via the following $R$ reactions
	\begin{equation}\label{reactions}
		y_{r,1}\S_1 + \cdots + y_{r,m}\S_m \xrightarrow{k_r(x,t)} y_{r,1}'\S_1 + \cdots + y_{r,m}'\S_m, \quad r = 1,\ldots, R,
	\end{equation}
	where $k_r(x,t)$ are the \textit{reaction rate coefficients} whose value depends on the spatial variable $x\in\Omega$ and time $t\in \R_+$, and $y_{r,i}, y_{r,i}' \in \{0\}\cup [1,\infty)$ are stoichiometric coefficients. Denoting by $y_r = (y_{r,i})_{i=1,\ldots, m}$ and $y_r' = (y_{r,i}')_{i=1,\ldots, m}$, $r=1,\ldots, R$, the vector of stoichiometric coefficients, we can rewrite the reactions in \eqref{reactions} as
	\begin{equation}\label{reactions_short}
		y_r \xrightarrow{k_{r}(x,t)} y_r', \quad r = 1,\ldots, R.
	\end{equation}
	Assume that the reaction system takes place in a bounded vessel $\Omega\subset\R^n$ with Lipschitz boundary $\partial\Omega$. Let $u_i := u_i(x,t)$ be the concentration of $\S_i$ at position $x\in\Omega$ at time $t>0$. Assume that each species diffuses at a different rate. Then one can apply second Fick's law and the law of mass action to obtain the following reaction diffusion system for the vector of concentrations $u = (u_1,\ldots, u_m)$
	\begin{equation}\label{Sys}
		\left\{
		\begin{aligned}
			&\partial_t u_i - \div(D_i(x,t)\na u_i) = R_i(x,t,u):= \sum_{r=1}^Rk_r(x,t)(y_{r,i}' - y_{r,i})u^{y_r}, && x\in\Omega,\\
			&D_i(x,t)\na u_i\cdot \nu = 0, &&x\in\partial\Omega,\\
			&u_i(x,0) = u_{i,0}(x), &&x\in\Omega,
		\end{aligned}
		\right.
	\end{equation}
	for all $i=1,\ldots, m$, where
	\begin{equation*}
		u^{y_r} = \prod_{i=1}^mu_i^{y_{r,i}},
	\end{equation*}
	the diffusion coefficients $D_i: \Omega\times \R_+\to \R^{n\times n}$, $\nu(x)$ is unit outward normal vector at $x\in\partial\Omega$. Here the homogeneous Neumann boundary conditions indicate that the chemical system is isolated. Thanks to
 these conditions, there are possibly a number of conservation laws corresponding to \eqref{Sys}. Indeed, denoting
	\begin{equation*}
		\W = (y_r' - y_r)_{r=1,\ldots, R} \in \R^{m\times R},
	\end{equation*}
	and $K := \dim(\ker(\W^\top))$, we let $q_1, \ldots, q_K$ be the column vectors forming a basis of $\ker(\W^\top)$. Then from the system \eqref{Sys} we have, \textit{formally}, for any $1\le j\le K$
	\begin{equation*}
		\frac{d}{dt} \intO q_j\cdot udx = \blue{\sum_{r=1}^R\intO k_r(x,t)(q_j\cdot(y_r' - y_r))u^{y_r}dx} = 0,
	\end{equation*}
	which lead to $K$ linearly independent conservation laws
	\begin{equation}\label{conservation_laws}
		\intO q_j\cdot u(x,t)dx = \intO q_j\cdot u_0(x)dx, \quad \forall j=1,\ldots, K,
	\end{equation}
	where $u_0 := (u_{1,0},\ldots, u_{m,0})$. By some rescaling, we can assume that $\Omega$ has volume one, i.e. $|\Omega| = 1$, \textbf{which we will assume throughout this paper}.
	
	\subsection{General systems with degenerate reactions}
	
	We consider first the case when (only) the reactions happen in a subdomain of $\Omega$. In order to set up the problem, we use a graph representation of the reaction network \eqref{reactions}. Let $\V = \{y_r, y_r'\}_{r=1,\ldots, R}\subset \R_+^{m}$ be the set of vertices (a set of points in $\R_+^m$). The set of directed edges are the reactions in \eqref{reactions}, $\Ed = \{y_r \to y_r': r= 1,\ldots, R \}$. We remark that for convenience, if a letter, say $y$, denotes the reactant complex in a reaction, then the corresponding letter $y'$ denotes the product complex. Then $\G = (\V,\Ed)$ forms a directed graph. A subset of vertices $\mathsf{U} \subset \V$ is called \textit{strongly connected}, if for any $v_1\ne v_2 \in \mathsf{U}$, there exists a \blue{sequences of vertices} $v_1=: w_1 \to w_2 \to \cdots \to w_r := v_2$, $r\ge 2$, where $w_j \to w_{j+1}\in \Ed$ for all $j=1,\ldots, r-1$. It is noted that in the language of chemical reaction network theory, these strongly connected components are usually called \textit{linkage classes}, see e.g. \cite{anderson2011proof}. A classical result in graph theory implies that $\G$ can be decomposed into strongly connected components. In this paper, we assume the following:
	\begin{equation}\label{A}\tag{A}
		\begin{tabular}{p{13cm}}
			There is no edge between any two different strongly connected components of the graph $\G$.
		\end{tabular}
	\end{equation}
	It is remarked that even though the components are disconnected, they do not possess their own decoupled dynamics since the chemical species can be present in all components (see Figure \ref{fig1}).
	Denote by $s\ge 1$ the number of strongly connected components of $\G$, i.e. $\G$ consists of the components $\C_1,\ldots, \C_s$. Thanks to assumption \eqref{A}, without loss of generality, we can re-label the vertices of $\G$ such that there exist $L_0 = 0 < L_1 < \ldots < L_{s-1} < \blue{L_s = R}$ with the property: for any $1\le l \le s$ the reactions $y_j \to y_j'$ for $L_{l-1}+1 \le j \le L_l$ form the $l$-th component.
	
	\medskip
	A subset $A\subset\mathbb R^n$ is said to satisfy an assumption \eqref{P} if the Poincar\'e-Wirtinger inequality in $A$ holds, i.e. there exists a constant $C_A$ depending only on $A$ such that
	\begin{equation}\label{P}\tag{P}
		\|\na u\|_{L^2(A)}^2 \geq C_A\left\|u - \frac{1}{|A|}\int_{A}u(x)dx\right\|_{L^2(A)}^2 \quad \forall u \in H^1(A).
	\end{equation}
	
	The next assumption is concerning the case when the partial reactions take place in positive measured sets.
	\begin{equation}\label{B}\tag{B}
		\begin{tabular}{p{14cm}}
			For each $l\in \{1,\ldots, s\}$, there is a function $\alpha_l:\Omega\times \R_+ \to \R_+$ such that $k_j(x,t) = \beta_j \alpha_l(x,t)$ for
			some $\beta_j>0$ and for $j=L_{l-1}+1, \ldots, L_l$. Moreover, there exists a subset $\omega_l \subset \Omega$
			with $|\omega_l| > 0$ satisfying (\ref{P}), and a positive number $\underline{\alpha}_l>0$ such that $\alpha_l(x,t)\ge \underline{\alpha}_l$ for a.e. $x\in\omega_l$.
		\end{tabular}
	\end{equation}
	
	Roughly speaking, assumption \eqref{B} means in particular that for each strongly connected component $\C_l$, all of its reaction rate coefficients scale with a function $\alpha_l$. This is important to define a complex balanced equilibrium to \eqref{Sys} (see Definition \ref{def:equilibrium}). The lower bound assumption of $\alpha_l$ means that there is a positively measured set $\omega_l$, which satisfies \eqref{P}, where all reactions of the component $\C_l$ happen. A specific example for \eqref{B} is when the sets $\{x\in\Omega: \alpha_l(x) \ge \underline{\alpha}_l\}$ are open for all $l=1,\ldots, s$. We present in Figure \ref{fig1} (a) an example where \eqref{A} and \eqref{B} are satisfied. Of particular physical relevance is the case when each component consists of a reversible reaction. This happens, for instance, when each of these reversible reactions requires a certain catalysis, which is present only in a subset of the medium, see Figure \ref{fig1} (b) for such a situation.
	
	
	\begin{figure}
		\begin{center}
			\begin{tikzpicture}[scale=1.1, smooth cycle]
				\draw[thick] plot[smooth cycle] coordinates{(1,0.5) (3,0) (5.5, 1) (6.5,2) (5,3.5) (3,4) (0,3)}; 
				
				\draw[thick,draw=none,scale=1.2,fill=green,yshift=1.1cm,xshift=0.15cm]plot[tension=.6] coordinates{(0,0.8) (1.5, -0.3) (3,0.7) (1.5,1.7)}; 
				
				\draw[thick,draw=none,fill=yellow,yshift=1.2cm,xshift=3.7cm] plot coordinates{(0.5,2) (2.4,1) (1.9,0) (0.5,0)}; 
				
				\node (A) at (3,0.2) {$\Omega$};
				\node (B) at (2.7,3.3) {$\omega_1$};
				\node (C) at (4.8,3.2) {$\omega_2$};

				\node (a) at (0.77, 2.7) {$\S_1$}; 
				\node (b) at (2.77,2.7) {$\S_2+\S_3$};
				\node (c) at (2.2,1.2) {$2\S_2$};
				\draw[thick,arrows=->] ([xshift =0.5mm]a.east) -- node [above] {\scalebox{.8}[.8]{$k_1(x,t)$}} ([xshift =-0.5mm]b.west);
				\draw[thick,arrows=->] ([yshift=0.5mm]c.west) -- node [left] {\scalebox{.8}[.8]{$k_3(x,t)$}} ([yshift=-0.5mm]a.south);
				\draw[thick,arrows=->] ([xshift =-0.5mm,yshift=-0.5mm]b.south) -- node [right] {\scalebox{.8}[.8]{$k_2(x,t)$}} ([yshift =0.5mm]c.east);
				
				\node (d) at (4.3,2.1) {$2\S_1$};
				\node (e) at (5.8,2.1) {$2\S_3$};
				\draw[thick,arrows=->] ([xshift =0.5mm,yshift=0.4mm]d.east) -- node [above] {\scalebox{.8}[.8]{$k_4(x,t)$}} ([xshift =-0.5mm,yshift=0.4mm]e.west);
				\draw[thick,arrows=->] ([xshift =-0.5mm,yshift=-0.4mm]e.west) -- node [below] {\scalebox{.8}[.8]{$k_5(x,t)$}} ([xshift =0.5mm,yshift=-0.4mm]d.east); 
			\end{tikzpicture}
			\begin{tikzpicture}[scale=1.1, smooth cycle]
				\draw[thick] plot[smooth cycle] coordinates{(1,0.5) (3,0) (5.5, 1) (6.5,2) (5,3.5) (3,4) (0,3)}; 
				
				\draw[thick,draw=none,scale=1.2,fill=cyan,yshift=1.1cm,xshift=0.15cm]plot[tension=.6] coordinates{(0,0.8) (1.5, -0.3) (3,0.7) (1.5,1.7)}; 
				
				\draw[thick,draw=none,fill=yellow,yshift=1.2cm,xshift=3.7cm] plot coordinates{(0.5,2) (2.4,1) (1.9,0) (0.5,0)}; 
				
				\node (A) at (3,0.2) {$\Omega$};
				\node (B) at (2.7,3.3) {$\omega_1$};
				\node (C) at (4.8,3.2) {$\omega_2$};

				\node (a) at (0.77, 2.2) {$\S_1$}; 
				\node (b) at (2.77,2.2) {$\S_2+\S_4$};
				\draw[thick,arrows=->] ([xshift =0.5mm,yshift=0.4mm]a.east) -- node [above] {\scalebox{.8}[.8]{$k_6(x,t)$}} ([xshift =-0.5mm,yshift=0.4mm]b.west);
				\draw[thick,arrows=->] ([xshift =-0.5mm,yshift=-0.4mm]b.west) -- node [below] {\scalebox{.8}[.8]{$k_7(x,t)$}} ([xshift =0.5mm,yshift=-0.4mm]a.east);
				
				\node (d) at (4.3,2.1) {$2\S_1$};
				\node (e) at (5.8,2.1) {$2\S_3$};
				\draw[thick,arrows=->] ([xshift =0.5mm,yshift=0.4mm]d.east) -- node [above] {\scalebox{.8}[.8]{$k_4(x,t)$}} ([xshift =-0.5mm,yshift=0.4mm]e.west);
				\draw[thick,arrows=->] ([xshift =-0.5mm,yshift=-0.4mm]e.west) -- node [below] {\scalebox{.8}[.8]{$k_5(x,t)$}} ([xshift =0.5mm,yshift=-0.4mm]d.east); 
			\end{tikzpicture}
		\end{center}
		\begin{tabular}{cc}
			(a)\hspace*{3.5cm} & \hspace*{3.5cm}(b)
		\end{tabular}
		\caption{Example of a complex balanced network satisfying \eqref{A} and \eqref{B}.\\
			$\bullet$ In (a), there are two strongly connected components $\C_1 = \{\S_1, \S_2 + \S_3, 2\S_2\}$ and $\C_2 = \{2\S_1, 2\S_3\}$ together with corresponding reactions, and the reactions in these components happen in open sets $\omega_1$ and $\omega_2$, respectively. Here for $i\in \{1,2,3\}$,  $k_i(x,t) = \beta_i\alpha_1(x,t)$ with $\alpha_1(x,t)\ge \underline{\alpha}_1 > 0$ in $\omega_1$, and for $j\in\{4,5\}$, $k_j(x,t) = \beta_j\alpha_2(x,t)$ with $\alpha_2(x,t) \ge \underline{\alpha}_2>0$ in $\omega_2$.  		
			Note that the chemicals $\S_1$ and $\S_3$ appear both on $\omega_1$ and $\omega_2$, and therefore the dynamics of the whole system couples the reactions in both of these subdomains.\\
			$\bullet$ In (b) each component has a single reversible reaction. This could be physically relevant, for instance, when the reversible reactions $\S_1 \leftrightarrows \S_2 + \S_4$ and $2\S_1 \leftrightarrows 2\S_3$ require certain catalysts to happen and these catalysts are only present in $\omega_1$ and $\omega_2$ respectively.}\label{fig1} 
	\end{figure}
	
	\begin{definition}\label{def:equilibrium}
		A spatially homogeneous state $u_\infty = (u_{1,\infty},\ldots, u_{m,\infty})\in \R_+^m$ is called a complex balanced equilibrium (CBE for short) for the system \eqref{Sys} if for any $l=1,\ldots, s$ and any $y\in \C_l$, the following equality holds
		\begin{equation}\label{equilibrium_cond}
			\blue{u_\infty^{y}\underset{y_j=y}{\sum_{L_{l-1}+1\le j\le L_l}}\beta_j =} \underset{y_j=y}{\sum_{L_{l-1}+1\le j\le L_l}}\beta_j u_\infty^{y_j} = \underset{y_k'=y}{\sum_{L_{l-1}+1\le k\le L_l}}\beta_ku_\infty^{y_k}.
		\end{equation}
	\end{definition}
	Thanks to assumptions \eqref{A} and \eqref{B}, it follows that $R_i(u_\infty) = 0$ for all $i=1,\ldots, m$, which means that $u_\infty$ a spatially homogeneous steady state of system \eqref{Sys}.
	\begin{remark}
		Consider the reversible reaction $\S_1 \underset{k_1(x)}{\overset{k_2(x)}{\leftrightarrows}} \S_2$, where the reaction rate coefficients $k_1, k_2$ depend only on $x\in\Omega$, which results in the reaction-diffusion system
		\begin{equation}\label{2x2}
			\begin{cases}
				\partial_t u_1 - \na\cdot(D_1(x,t)\nabla u_1) = -k_1(x)u_1 + k_2(x)u_2,\\
				\partial_t u_2 - \na\cdot(D_2(x,t)\nabla u_2) = k_1(x)u_1 - k_2(x)u_2,\\
				D_i(x,t)\nabla u_i \cdot \nu = 0,\quad i=1,2,\\
				u_{i}(x,0) = u_{i,0}(x),\quad i=1,2.
			\end{cases}
		\end{equation}
		When $k_1$ and $k_2$ are strictly positive constants, then the network is obviously complex balanced and there is a unique positive equilibrium for each positive initial total mass. However, if $\text{supp}(k_1)\cap \text{supp}(k_2) = \emptyset$, then \eqref{B} is violated and the only spatially homogeneous steady state of \eqref{2x2} is the zero state $(0,0)$.
		
	\end{remark}
	It is also remarked that in the case when the functions $k_r$ are constants, the CBE $u_\infty$ in Definition \ref{def:equilibrium} coincides with the classical definition in chemical reaction network theory, see e.g. \cite{feinberg2019foundations}. It can be also seen from Definition \ref{def:equilibrium} that the set of CBE forms a manifold (possibly with singularities) in $\R_+^m$. To uniquely determine $u_\infty$, we need the conservation laws \eqref{conservation_laws}.
	
	\begin{lemma}\cite{feinberg2019foundations}\label{lem:equilibrium}
		Assume assumptions \eqref{A} and \eqref{B}. If there exists a CBE $u_\infty$ as in Definition \ref{def:equilibrium}, then any spatially homogeneous steady state of \eqref{Sys} is complex balanced. Moreover, for any non-negative initial data $u_0\in L^1_+(\Omega)^m$, there exists a unique \textbf{strictly positive} CBE $u_\infty \in (0,\infty)^m$ satisfying \eqref{equilibrium_cond} and the conservation laws
		\begin{equation*}
			q_j\cdot u_\infty = \intO q_j\cdot u_0(x)dx, \quad \forall j=1,\ldots, K
		\end{equation*}
		where $(q_j)_{j=1,\ldots, K}$ is defined in \eqref{conservation_laws}. It is remarked that there \blue{might exist} (possibly infinitely) many \textbf{boundary} CBE which lie on $\partial\R_+^m$.
	\end{lemma}
	Due to Lemma \ref{lem:equilibrium} we will refer to the strictly positive CBE simply as CBE, and the boundary CBE as boundary equilibria. Our first main result of this paper is the exponential convergence to equilibrium for the system \eqref{Sys} under assumptions \eqref{A}, \eqref{B}, and the fact that there are no boundary equilibria.
	
	\begin{theorem}\label{thm:main1}
		Assume the following
		\begin{itemize}
			\item[(i)] \eqref{A} and \eqref{B};
			\item[(ii)] the diffusion matrices are symmetric and bounded, i.e. $D_i\in L^{\infty}_{\text{loc}}(\R_+;L^{\infty}(\Omega; \R_{\text{\normalfont sym}}^{n\times n}))$, and 
			\begin{equation*}
				\xi^\top D_i(x,t)\xi \ge \underline{D}_i|\xi|^2, \quad \forall \xi \in \R^m, \text{ a.e. } (x,t)\in\Omega\times \R_+
			\end{equation*}
			for some $\underline{D}_i>0$, for all $i=1,\ldots, m$;
			\item[(iii)] there exists a CBE to \eqref{Sys} as defined in Definition \ref{def:equilibrium};
			\item[(iv)] there are no boundary equilibria.
		\end{itemize}
		Then for any non-negative initial data $u_0 \in L^1_+(\Omega)^m$ such that $\sum_{i=1}^{m}\intO u_{i,0}|\log u_{i,0}|dx < +\infty$, there exists a global renormalised solution to \eqref{Sys} as in Definition \ref{def:renormalised_solution} below. Moreover, all renormalised solutions converge exponentially to CBE with an exponential rate, i.e.
		\begin{equation*}
			\sum_{i=1}^m\|u_i(t) - u_{i,\infty}\|_{L^1(\Omega)} \le Ce^{-\lambda t}, \quad \forall t\ge 0.
		\end{equation*}
	\end{theorem}
	
	It is emphasised that, in general, the global existence theory for \eqref{Sys} is highly non-trivial due to the possible arbitrarily high orders of the nonlinearities, see e.g. \cite{pierre2010global} for an extensive survey. If \eqref{Sys} possesses an entropic dissipation structure, which is a consequence of having a CBE, the only known concept of global solution to \eqref{Sys} is \textit{renormalised solutions}, see e.g. \cite{fischer2015global}, which has minimal regularity, which in turns makes the study of their dynamics highly challenging. In order to prove Theorem \ref{thm:main1}, we use the \blue{entropy method, which was widely used in kinetic theory  and other fields in the 90s (cf. for example \cite{CV} ), and later extended to chemical reaction-diffusion systems \cite{Desvillettes2006,desvillettes2007entropy,Mielke2015,mielke2018convergence,Haskovec2018}. An important feature of this method is that it} relies on functional inequalities and consequently requires minimal regularity of solutions, see e.g. \cite{fellner2018convergence}. This is in contrast to that of \cite{desvillettes2021exponential} and therefore it allows us to show the equilibration of \textit{all} renormalised solutions. More precisely, we consider the following relative entropy
	\begin{equation*}
		\E(u|\uinf)= \sum_{i=1}^{m}\int_{\Omega}\bra{u_i\log\frac{u_i}{u_{i,\infty}} - u_i + u_{i,\infty}}dx,
	\end{equation*}
	and the corresponding entropy dissipation
	\begin{equation*}
		\D(u) = \sum_{i=1}^{m}\int_{\Omega}D_i(x,t)\nabla u_i\cdot\frac{\nabla u_i}{u_i}dx + \sum_{r=1}^{R}\int_{\Omega}k_r(x,t)\uinf^{y_{r}}\Psi\bra{\frac{u^{y_r}}{\uinf^{y_r}}; \frac{u^{y_r'}}{\uinf^{y_r'}}}dx,
	\end{equation*}
	where the function $\Psi$ is defined as
	\begin{equation*}
		\Psi(w;z) = w\log\frac wz - w + z.
	\end{equation*}
	\textit{Formally}, one expects the entropy-entropy dissipation law, see \cite[Proposition 2.1]{desvillettes2017trend}
	\begin{equation}\label{eedl}
		\frac{d}{dt}\E(u|u_\infty) = -\D(u).
	\end{equation}
	The cornerstone of the entropy method is to show the following functional inequality
	\begin{equation*} 
		\boxed{\D(u)\gtrsim \E(u|u_\infty)} \quad \forall u: \Omega \to \R_+^m \text{ satisfying the conservation laws } \eqref{conservation_laws}.
	\end{equation*}
	To overcome the difficulty stemming from degeneracy of the reactions, 
	our key idea is to control the reaction terms in the entropy dissipation by their partial averages in corresponding subdomains where reactions happen, and then to estimate the differences by using the diffusion of all species.

	\subsection{A specific case of even more degenerate situations}
	
	In the proof of Theorem \ref{thm:main1}, it is of importance that the diffusion is non-degenerate and the reaction happens in a subdomain which has certain regularity, e.g. Lipschitz boundary, or contains an open domain. The latter allows us to apply the Poincar\'e inequality \eqref{P} in (a subset of) the subdomain, which then combines with the reaction to drive the trajectory eventually to the spatially homogeneous equilibrium. Due to the low regularity of renormalised solutions, relaxing or weakening these assumptions seem difficult. However, if the solution is known to be bounded uniformly in time, it might be possible to handle degenerate diffusion as well as reaction in much rougher subdomains, namely subsets of $\Omega$ which are only {\it measurable} with positive measure. A key idea in these situations is that using the uniform boundedness of solutions, we can estimate many quantities, e.g. the relative entropy, \textit{pointwise} rather just through integrals. We illustrate this by studying the following reversible reactions
	\begin{equation*} 
		\S_1 \underset{k_1(x,t)}{\overset{k_1(x,t)}{\leftrightarrows}} 2\S_2, \quad \S_2 \underset{k_2(x,t)}{\overset{k_2(x,t)}{\leftrightarrows}} 2\S_3,
	\end{equation*}
	which result in the reaction-diffusion system
	\begin{equation}\label{special_sys}
		\begin{cases}
			{\partial}_{t}u_{1}-\div(d_1(x,t)\na u_1)=k_{1}(x,t)\left(  u_{2}^{2}-u_{1}\right)
			\text{ ,} & \text{in}~\Omega\times\R_+  \text{,}\\
			{\partial}_{t}u_{2}-\div(d_2(x,t)\na u_2)=-2k_{1}(x,t)\left(  u_{2}^{2}-u_{1}\right)
			+k_{2}(x,t)\left(  u_{3}^{2}-u_{2}\right), & \text{in}~\Omega
			\times\R_+,\\
			{\partial}_{t}u_{3}-\div\left(d_{3}(x,t)\nabla u_{3}\right)  =-2k_{2}(x,t)%
			\left(  u_{3}^{2}-u_{2}\right), & \text{in}~\Omega\times\R_+,\\
			d_1(x,t)\nabla u_1 \cdot \nu = d_2(x,t)\nabla u_2 \cdot \nu = d_3(x,t)\nabla u_3 \cdot \nu = 0, & \text{on}~\partial\Omega\times\R_+,\\
			u_{1}\left(  \cdot,0\right)  =u_{1,0}, \quad u_{2}\left(  \cdot,0\right)
			=u_{2,0},\quad  u_{3}\left(  \cdot,0\right)  =u_{3,0}, & \text{in}~\Omega.%
		\end{cases}
	\end{equation}
	Here $0\le k_{1}, k_{2}\in
	L^{\infty}\left(  \Omega\times\left(  0,+\infty\right)  \right) $ are reaction rate coefficients. It is easy to check that solutions to \eqref{special_sys} satisfy the conservation law
	\begin{equation}\label{conservation_special}
		\int_{\Omega}(4u_1(x,t) + 2u_2(x,t) + u_3(x,t))dx = \int_{\Omega}(4u_{1,0}(x) + 2u_{2,0}(x) + u_{3,0}(x))dx =: M,
	\end{equation}
	for all $t$ where the solutions exist. It is easy to see that for any positive initial mass $M$ defined in \eqref{conservation_special} there exists a unique strictly positive equilibrium $u_{\infty} = (u_{1,\infty}, u_{2,\infty}, u_{3,\infty})$ which solves
	\begin{equation}\label{equi_sys_special}
		\begin{cases}
			u_{2,\infty}^2 = u_{1,\infty}, \\
			u_{3,\infty}^2 = u_{2,\infty},\\
			4u_{1,\infty} + 2u_{2,\infty} + u_{3,\infty} = M,
		\end{cases}
	\end{equation}
	since $x \mapsto 4x^4 + 2 x^2 + x$ is strictly increasing on $\R_+$.
	
	\medskip
	Let $\omega_1$ and $\omega_2$ be non-empty subsets of $\Omega$. 
	To study the convergence to equilibrium for \eqref{special_sys}, we assume that there is some $\kappa>0$ such that
	\begin{equation}\label{k1k2}
		\begin{aligned}
			k_{1}\left(  x,t\right)  \geq \kappa \quad \forall (x,t)\in\omega_{1}\times\R_+,\\
			k_{2}\left(  x,t\right)  \geq \kappa \quad \forall (x,t)\in\omega_{2}\times\R_+.
		\end{aligned}
	\end{equation}
	For the (scalar) diffusion coefficients $d_i$, we assume that
	\begin{equation}\label{D1D2D3}
		d_i \in L_{\text{loc}}^\infty(\R_+;L^{\infty}(\Omega)), \quad i=1,2,3,
	\end{equation}
	and that there is some $\delta>0$ such that
	\begin{equation}\label{D1D2}
		d_1(x,t) \ge \delta \quad \text{and} \quad d_2(x,t) \ge \delta, \quad \forall (x,t)\in\Omega\times\R_+.
	\end{equation}
	We also assume that there exists $\delta_0\ge 0$ with
	\begin{equation}\label{D3}
		\begin{aligned}
			d_3(x,t) &\geq \delta_0, && \text{a.e. } (x,t)\in\Omega\times\R_+.
		\end{aligned}
	\end{equation}
	
	In the first case, we consider $\delta_0>0$, meaning that $\S_1, \S_2, \S_3$ have full diffusion, but the sets $\omega_1$ and $\omega_2$ where reactions happen are \textit{only measurable} with positive measures.
	
	\begin{theorem}\label{thm:measurable}
		Assume \eqref{k1k2}, \eqref{D1D2D3}, \eqref{D1D2} and \eqref{D3} with $\delta_0>0$ and $\omega_1, \omega_2$ are {\normalfont measurable sets} with positive measures. Then for any non-negative, bounded initial data $u_{0} \in L^{\infty}_+(\Omega)^3$, there exists a unique non-negative, weak solution to \eqref{special_sys}, which is bounded uniformly in time, i.e.
		\begin{equation}\label{sup-norm-bound-1}
			\sup_{t\ge 0}\sup_{i=1,\ldots, 3}\|u_i(t)\|_{L^{\infty}(\Omega)} \le \mathcal{C}_0<+\infty.
		\end{equation}
		Moreover, this solution converges exponentially fast, i.e. there are {\normalfont explicitly computable} constants $C, \lambda>0$ such that
		\begin{equation*} 
			\sum_{i=1}^3\|u_i(t) - u_{i,\infty}\|_{\LO{1}}^2 \le Ce^{-\lambda t}, \quad \forall t>0,
		\end{equation*} 
		where $u_\infty$ solves \eqref{equi_sys_special}.
	\end{theorem}

	\begin{remark}\hfill
		\begin{itemize}
			\item The convergence rate to equilibrium depends on $\omega_1$ and $\omega_2$ in the following way
			\begin{equation*} 
				\frac 1\lambda =C\bra{\frac{ 1}{|\omega_1|}+\frac{ 1}{|\omega_2|}},
			\end{equation*}
			where the constant $C>0$ is independent of $\omega_1$ and $\omega_2$. It is clear that with this relation, $\lambda\to 0$ if either $|\omega_1| \to 0$ or $|\omega_2| \to 0$. 
			\item By interpolating the exponential convergence with the $\LO{\infty}$-bound \eqref{sup-norm-bound-1}, one can immediately get exponential convergence to equilibrium in $\LO{p}$ for any $1<p<\infty$. Convergence in stronger norms is possible to obtain when the functions $d_i$ and $k_i$ are sufficiently smooth.
		\end{itemize}
	\end{remark}

	\medskip
	If the diffusion of $\S_3$ is completely degenerate, i.e. $d_3(x,t) = 0$ for all $(x,t)\in \Omega\times\R_+$, the global existence of solutions can be obtained with bounded initial data, see e.g. \cite{einav2020indirect} or \cite{braukhoff2022quantitative}. However, if $\{(x,t): d_3(x,t) = 0\}$ has strictly positive measure but is not zero on $\Omega\times\R_+$, the global existence of solutions to \eqref{special_sys} is nontrivial, see e.g. \cite{desvillettes2007global}. In our second example, we prove the global existence and convergence to equilibrium in the case when the diffusion of $\S_3$ is not ``too'' degenerate, i.e. $d_3$ depends only on $x\in\Omega$ and vanishes only on a zero measure set, that is
	\begin{equation}\label{D3'}
		\begin{aligned}
			|\{x\in\Omega: d_3(x) = 0\}| = 0.
		\end{aligned}
	\end{equation}	
	\begin{theorem}\label{thm:main3}
		Suppose that \eqref{k1k2}, \eqref{D1D2D3}, \eqref{D1D2} and \eqref{D3'} hold. Moreover, assume that $d_3 \in W^{1,q}(\Omega)$ for some $q>\max(n,2)$.	Then for any non-negative, bounded initial data $u_{0} \in L^{\infty}_+(\Omega)^3$, there exists a unique global non-negative weak solution to \eqref{special_sys}, which is bounded uniformly in time, i.e.
		\begin{equation}\label{sup-norm-bound}
			\sup_{t\ge 0}\sup_{i=1,\ldots, 3}\|u_i(t)\|_{L^{\infty}(\Omega)} \le \mathcal{C}_0<+\infty.
		\end{equation}
		
		\medskip
		Moreover, assume that $\omega_1$ is measurable with $|\omega_1|>0$, $\omega_2\subset \Omega$ is open with Lipschitz boundary, and
		\begin{equation}\label{assumption} 
			\{x\in\overline{\Omega}: d_3(x) =0\} \subset \omega_2.
		\end{equation} 
		Then the weak solution to \eqref{special_sys} converges to the equilibrium exponentially fast, i.e. there exist {\normalfont explicitly computable} strictly positive constants $C, \lambda>0$ such that
		\begin{equation}\label{exp_convergence}
			\sum_{i=1}^3\|u_i(t) - u_{i,\infty}\|_{L^{1}(\Omega)} \le Ce^{-\lambda t}, \quad \forall t\ge 0,
		\end{equation}
		where $\uinf$ solves \eqref{equi_sys_special}.
	\end{theorem}

	\begin{remark}\hfill
		\begin{itemize}
			\item Assumption \eqref{assumption} and the fact that $\omega_2$ is open imply that there is an open set where a strictly positive diffusion of $\S_3$ and the reaction $\S_2 \leftrightarrows 2\S_3$ are both present. This will be used crucially in our proof.
			\item When \eqref{D3'} is not satisfied, i.e. $d_3$ can be zero on a set of positive measure, the global existence of solutions to \eqref{special_sys} is unclear, see e.g. \cite{desvillettes2007global}. Nevertheless, by replacing $d_3$ by $d_3 + \eps$, we can use the same arguments in the proof of Theorem \ref{thm:main3} to show that the solution to \eqref{special_sys} (with $d_3$ replaced by $d_3 +\eps$) converges to equilibrium exponentially with rates and constants {\normalfont independent of $\eps>0$}. The global existence of solutions and convergence to equilibrium for \eqref{special_sys} in case $|\{x\in\Omega: d_3(x) = 0\}| > 0$ remains as an interesting open problem.
		\end{itemize}	
	\end{remark}	

	
	\medskip
	\textbf{Notation}. We use the following notation in this paper:
	\begin{itemize}
		\item for a measurable set $A$ with positive measure, $[u]_A$ denotes the spatial average of $u$ over $A$, 
		\begin{equation*}
			[u]_A: = \frac{1}{|A|}\int_{A}u(x)dx;
		\end{equation*}
		\item we use capital letters to denote the square roots of the corresponding letters, e.g.
		\begin{equation*}
			U_i = \sqrt{u_i}, \quad U_{i,\infty} = \sqrt{u_{i,\infty}};
		\end{equation*}
		\item the notation $X\lesssim Y$ means that there exists $C>0$ independent of $X$ and $Y$ such that $X \le CY$. Occasionally, we write $X\lesssim_{\alpha,\beta,\ldots} Y$ to emphasise the dependence of the inequality on the parameters $\alpha, \beta, \ldots$
		\item for a positive vector $u\in (0,\infty)^m$ and $y \in \R^m$,
		\begin{equation*}
			u^y:= \prod_{i=1}^mu_i^{y_i}.
		\end{equation*}
	\end{itemize}
	
	\medskip
	\textbf{Organization of the paper.} In the next section, we will prove Theorem \ref{thm:main1} for degenerate reaction. The convergence to equilibrium with reactions happening  in measurable sets (Theorem \ref{thm:measurable}) will be shown in Section \ref{sec3}. Finally, Section \ref{sec4} considers \eqref{special_sys} with both degenerate diffusion and reactions as stated in Theorem \ref{thm:main3}.
	\section{Degenerate reactions - Proof of Theorem \ref{thm:main1}}\label{sec2}
	
	We start with the definition of renormalised solutions to \eqref{Sys}.
	\begin{definition}\label{def:renormalised_solution}
		A vector of concentration $u: \Omega\times \R_+ \to \R_+^m$ is called a global renormalised solution to \eqref{Sys}  if for any $T>0$, $u_i\log u_i \in L^\infty(0,T;L^1(\Omega))$, $\sqrt{u_i} \in L^2(0,T;H^1(\Omega))$ and for any smooth function $\xi: \R_+^m\to \R$ with compactly supported derivative $\na \xi$ and every $\psi \in C^\infty(\overline{\Omega}\times \R_+)$, there holds
		\begin{align*}
			&\intO \xi(u(\cdot, T))\psi(\cdot,T)dx - \intO\xi(u_0)\psi(\cdot,0)dx - \int_0^T\intO \xi(u)\partial_t\psi dxdt\\
			&= -\sum_{i,j=1}^m \int_0^T\intO \psi \partial_i\partial_j \xi(u)(D_i(x,t)\na u_i)\cdot \na u_jdxdt\\
			&\quad - \sum_{i=1}^m\int_0^T\intO \partial_i\xi(u)(D_i(x,t)\na u_i)\cdot \na \psi dxdt + \sum_{i=1}^m\int_0^T\intO \partial_i\xi(u)R_i(x,t,u)\psi dxdt.
		\end{align*}	
	\end{definition}

	\begin{proposition}\label{pro:existence}
		Assume (i)--(iv) in Theorem \ref{thm:main1}. Then for any non-negative initial data $(u_{i,0})\in L^1_+(\Omega)^m$ with $\sum_{i=1}^m\intO u_{i,0}\log u_{i,0}dx < +\infty$, there exists a global renormalised solution to \eqref{Sys}.
	\end{proposition}
	\begin{proof}
		We will apply \cite[Theorem 1]{fischer2015global}. In order to do that, it is sufficient to check that 
		\begin{equation*}
			\sum_{i=1}^mR_i(x,t,u)\log\frac{u_i}{u_{i,\infty}} \le 0.
		\end{equation*}
		We use the ideas in \cite[Proposition 2.1]{desvillettes2017trend}. We rewrite, by using $R(x,t,u) := (R_i(x,t,u))$ and $\log(u/u_\infty) := (\log(u_i/u_{i,\infty}))$ for $i=1,..,m$,
		\begin{align*}
			&\sum_{i=1}^mR_i(x,t,u)\log\frac{u_i}{u_{i,\infty}} = R(x,t,u)\cdot \log\frac{u}{u_\infty}= \sum_{r=1}^Rk_r(x,t)u^{y_r}(y_r'-y_r)\cdot \log\frac{u}{u_\infty}\\
			&= -\sum_{r=1}^Rk_r(x,t)u^{y_r}\log\frac{u^{y_r-y_r'}}{u_\infty^{y_r-y_r'}}\\
			&= -\sum_{r=1}^Rk_r(x,t)u_\infty^{y_r}\left[\frac{u^{y_r}}{\uinf^{y_r}}\log\bra{\frac{u^{y_r}}{\uinf^{y_r}}\bigg/\frac{u^{y_r'}}{\uinf^{y_r'}}} - \frac{u^{y_r}}{\uinf^{y_r}} + \frac{u^{y_r'}}{\uinf^{y_r'}} \right]- \sum_{r=1}^R\bra{k_r(x,t)u^{y_r} - k_r(x,t)u^{y_r'}\frac{\uinf^{y_r}}{\uinf^{y_r'}}}\\
			&\le -\sum_{r=1}^R\bra{k_r(x,t)u^{y_r} - k_r(x,t)u^{y_r'}\frac{\uinf^{y_r}}{\uinf^{y_r'}}}.
		\end{align*}
		It remains to show that the last sum vanishes. Using assumption \eqref{A}, we can write
		\begin{align*}
			\sum_{r=1}^R\bra{k_r(x,t)u^{y_r} - k_r(x,t)u^{y_r'}\frac{\uinf^{y_r}}{\uinf^{y_r'}}} &= \sum_{l=1}^s\sum_{j=L_{l-1}+1}^{L_l}\bra{k_j(x,t)u^{y_j} - k_j(x,t)u^{y_j'}\frac{\uinf^{y_j}}{\uinf^{y_j'}}}.
		\end{align*}
		Now thanks to assumption \eqref{B}, for each $l\in\{1,\ldots, s\}$,
		\begin{align*}
			\sum_{j=L_{l-1}+1}^{L_l}&\bra{k_j(x,t)u^{y_j} - k_j(x,t)u^{y_j'}\frac{\uinf^{y_j}}{\uinf^{y_j'}}} = \alpha_l(x,t)\sum_{j=L_{l-1}+1}^{L_l}\bra{\beta_ju^{y_j} - \beta_ju^{y_j'}\frac{\uinf^{y_j}}{\uinf^{y_j'}}}\\
			&= \alpha_l(x,t)\sum_{y\in \C_l}\left[\underset{y_j=y}{\sum_{L_{l-1}+1\le j\le L_l}}\beta_ju^{y_j} - \underset{y_k'=y}{\sum_{L_{l-1}+1\le k\le L_l}}\beta_ku^{y_k'}\frac{\uinf^{y_k}}{\uinf^{y_k'}} \right]\\
			&= \alpha_l(x,t)\sum_{y\in \C_l}\left[u^{y}\underset{y_j=y}{\sum_{L_{l-1}+1\le j\le L_l}}\beta_j -  \frac{u^y}{\uinf^y}\underset{y_k'=y}{\sum_{L_{l-1}+1\le k\le L_l}}\beta_k \uinf^{y_k} \right]\\
			&= \alpha_l(x,t)\sum_{y\in \C_l}\frac{u^y}{\uinf^y}\left[\underset{y_j=y}{\sum_{L_{l-1}+1\le j\le L_l}}\beta_j\uinf^{y_j} - \underset{y_k'=y}{\sum_{L_{l-1}+1\le k\le L_l}}\beta_k \uinf^{y_k} \right]\\
			&= 0,
		\end{align*}
		thanks to the definition of $\uinf$ in Definition \ref{def:equilibrium}.
	\end{proof}
	Due to the low regularity of renormalised solution, we can only prove a weak version of the entropy-entropy dissipation relation \eqref{eedl}. Moreover, it can also be shown that renormalised solutions satisfy the conservation laws \eqref{conservation_laws}. 
	\begin{lemma}\label{lem0}
		Any renormalised solution to \eqref{Sys} satisfies the following weak entropy-entropy dissipation law
		\begin{equation*}
			\E(u(s)|u_\infty)\bigg|_{s=\tau}^{s=T} + \int_\tau^T \D(u(s))ds \le 0, \quad \forall 0\le \tau < T,
		\end{equation*}
		and the conservation laws \eqref{conservation_laws}, i.e.
		\begin{equation*}
			\intO q_j\cdot u(x,t)dx = \intO q_j\cdot u_0(x)dx, \quad \forall j=1,\ldots, K, \; \forall t\ge 0.
		\end{equation*}
		Consequently, there is $M_0>0$ depending on $\E(u_0|u_\infty)$ such that
		\begin{equation}\label{L1_bound} 
			\sup_{t\ge 0}\sup_{i=1,\ldots,m}\|u_i(t)\|_{\LO{1}} \le M_0.
		\end{equation}
	\end{lemma}
	\begin{proof}
		From the proof of Proposition \ref{pro:existence}, we have
		\begin{equation*}
			-\sum_{i=1}^mR_i(x,t,u)\log\frac{u_i}{u_{i,\infty}} = \sum_{r=1}^{R}\int_{\Omega}k_r(x,t)\uinf^{y_{r}}\Psi\bra{\frac{u^{y_r}}{\uinf^{y_r}}; \frac{u^{y_r'}}{\uinf^{y_r'}}}dx.
		\end{equation*}
		The weak entropy-entropy dissipation law and the conservation laws then follow from \cite[Propositions 5 and 6]{fischer2017weak}. To show \eqref{L1_bound} we first note that by choosing $\tau = 0$ in the weak entropy-entropy dissipation law, we have in particular
		\begin{equation*} 
			\E(u(t)|u_\infty)\le \E(u_0|u_{\infty}) \quad \forall t\ge 0.
		\end{equation*}
		Using the elementary inequality $u_i\le u_i\log(u_i/u_{i,\infty}) + C$,
		for a constant $C$ depending only on $u_{i,\infty}$, we get \eqref{L1_bound} immediately.
	\end{proof}
	
	The following Csisz\'ar-Kullback-Pinsker type inequality shows that a decay to zero of the relative entropy implies the convergence to equilibrium for solutions in $L^1(\Omega)$-norm.
	\begin{lemma}\label{lem:CSK}
		There exists a constant $C_{\text{\normalfont CSK}}>0$ depending on $M_0$ (see Lemma \ref{lem0}), the domain $\Omega$, and the equilibrium $u_\infty$,  such that 
		the following inequality holds for any renormalised solution to \eqref{Sys}
		\begin{equation*} 
			\E(u|u_\infty) \ge C_{\text{\normalfont CSK}}\sum_{i=1}^m\|u_i - u_{i,\infty}\|_{L^1(\Omega)}^2.
		\end{equation*}
	\end{lemma}
	\begin{proof}
		It is straightforward to check that the relative entropy satisfies the additivity
		\begin{equation}\label{entropy_additive}
			\E(u|u_\infty) = \E(u|\avO{u}) + \E(\avO{u}|u_\infty),
		\end{equation}
		where 
		\begin{equation}\label{defel}
			\E(\avO{u}|u_\infty) := \sum_i  \bigg( [u_i]_{\Omega}\, \log \bigg( \frac{ [u_i]_{\Omega} }{ u_{i,\infty} } \bigg) - 
			[u_i]_{\Omega} +  u_{i,\infty}  \bigg).
		\end{equation}
		By applying Csisz\'ar-Kullback-Pinsker's inequality for bounded domains, see e.g. \cite[Proposition]{Fellner2016a}, we have
		\begin{equation*}
			\E(u|\avO{u}) = \sum_{i=1}^{m}\int_{\Omega}u_i\log\frac{u_i}{\avO{u_i}}dx\ge C_1\sum_{i=1}^m\|u_i - \avO{u_i}\|_{\LO{1}}^2 ,
		\end{equation*}
		where $C_1$ depends only on $\Omega$ and the spatial dimension $n$. On the other hand, using the elementary inequality $x\log(x/y)-x+y \ge (\sqrt x - \sqrt y)^2$ and the $L^1$-bound \eqref{L1_bound}, we can estimate
		\begin{align*} 
			\E(\avO{u}|u_\infty) \ge \sum_{i=1}^m|\sqrt{\avO{u_i}} - \sqrt{u_{i,\infty}}|^2 &\ge \sum_{i=1}^m\frac{1}{(\sqrt{M_0}+\sqrt{u_{i,\infty}})^2}|\avO{u_i} - u_{i,\infty}|^2\\
			&\ge C_2\sum_{i=1}^m\|\avO{u_i} - u_{i,\infty}\|_{L^1(\Omega)}^2.
		\end{align*}
		The proof of Lemma \ref{lem:CSK} is then completed with $C_{\text{\normalfont CSK}} := \min\{C_1,C_2\}/2$.
	\end{proof}
	
	A crucial tool for proving Theorem \ref{thm:main1} is the following functional inequality.
	\begin{proposition}\label{pro:funtional_inequality}
		Assume \text{\normalfont (i)}--\text{\normalfont (iv)} in Theorem \ref{thm:main1}. Then there exists a constant $\lambda>0$ depending on $\Omega, \uinf, L, y_r, y_r',\beta_j, |\omega_l|$ (see assumption \eqref{B}), such that 
		\begin{align*}
			\D(u) \ge \lambda \, \E(u|\uinf)
		\end{align*}
		for any non-negative function vector $u: \Omega \to \R_+^m$ satisfying $\sum_{i=1}^{m}\intO u_i\log u_i \le L < +\infty$ and the conservation laws \eqref{conservation_laws}.
	\end{proposition}
	To prove Proposition \ref{pro:funtional_inequality}, we start with some preliminary results. Recall the notation,
	\begin{equation*}
		U_i = \sqrt{u_i}, \; U = (U_1, \ldots, U_m), \quad U_{i,\infty} = \sqrt{u_{i,\infty}}, \quad \Uinf = (U_{1,\infty}, \ldots, U_{m,\infty}),
	\end{equation*}
	and for any measurable set $A$, 
	\begin{equation*}
		\left[f\right]_{A}:= \frac{1}{|A|}\int_{A}f(x)dx.
	\end{equation*}
	\begin{lemma}\label{lem1}
		We have the following bounds
		\begin{equation*}
			\avO{U_i^2} + \avo{U_i^2}{l} + \avO{U_i} + \avo{U_i}{l} \lesssim 1 ,
		\end{equation*}
		for all $i = 1,\ldots, m$ and all $l = 1,\ldots, s$.
	\end{lemma}
	\begin{proof}
		The estimates
		\begin{equation*}
			\avO{U_i^2} + \avo{U_i^2}{l} \lesssim 1
		\end{equation*}
		follow directly from \eqref{L1_bound}. By Cauchy-Schwarz  inequality
		\begin{equation*}
			\avO{U_i} = \int_{\Omega}U_i(x)dx \lesssim \bra{\int_{\Omega}U_i^2(x)dx}^{1/2} \lesssim 1 ,
		\end{equation*}
		and the estimate $\avo{U_i}{l} \lesssim 1$ can be obtained similarly.
	\end{proof}
	An immediate estimate is
	\begin{equation}\label{h1}
		\begin{aligned}
			&\sum_{i=1}^{m}\int_{\Omega}D_i(x,t)\nabla u_i\cdot \frac{\nabla u_i}{u_i}dx + \sum_{r=1}^{R}\int_{\Omega}k_r(x,t)\uinf^{y_{r}}\Psi\bra{\frac{u^{y_r}}{\uinf^{y_r}}; \frac{u^{y_r'}}{\uinf^{y_r'}}}dx\\
			&\quad = \sum_{i=1}^{m}\int_{\Omega}D_i(x,t)\nabla u_i\cdot \frac{\nabla u_i}{u_i}dx + \sum_{l=1}^{s}\sum_{j=L_{l-1}+1}^{L_l}\int_{\Omega}k_j(x,t)\uinf^{y_{j}}\Psi\bra{\frac{u^{y_j}}{\uinf^{y_j}}; \frac{u^{y_j'}}{\uinf^{y_j'}}}dx\\
			&\quad \gtrsim \sum_{i=1}^{m}\int_{\Omega}\frac{|\nabla u_i|^2}{u_i}dx + \sum_{l=1}^{s}\sum_{j=L_{l-1}+1}^{L_l}\int_{\omega_l}\uinf^{y_{j}}\Psi\bra{\frac{u^{y_j}}{\uinf^{y_j}}; \frac{u^{y_j'}}{\uinf^{y_j'}}}dx ,
		\end{aligned}
	\end{equation}
	where we use the assumptions \eqref{B}, (ii) in Theorem \ref{thm:main1}, and the rewriting
	\begin{equation*}
		D_i(x,t)\nabla u_i\cdot\frac{\nabla u_i}{u_i} = 4\nabla\sqrt{u_i}^{\top} D_i(x,t) \nabla \sqrt{u_i} \ge 4\underline{D}_i|\nabla \sqrt{u_i}|^2 = \underline{D}_i\frac{|\nabla u_i|^2}{u_i}.
	\end{equation*}
	\begin{lemma}\label{lem2}
		For each $l\in \{1,\ldots, s\}$, it holds
		\begin{equation*}
			\begin{aligned}
				\sum_{i=1}^m\intO\frac{|\na u_i|^2}{u_i}dx +  \sum_{j=L_{l-1}+1}^{L_l}\int_{\omega_l}\uinf^{y_{j}}\Psi\bra{\frac{u^{y_j}}{\uinf^{y_j}}; \frac{u^{y_j'}}{\uinf^{y_j'}}}dx\\
				\gtrsim\sum_{j=L_{l-1}+1}^{L_l} \left(\avo{\frac{U}{\Uinf}}{l}^{y_j} - \avo{\frac{U}{\Uinf}}{l}^{y_j'}\right)^2.
			\end{aligned}
		\end{equation*}
		Here we recall the notation
		\begin{equation*}
			\avo{\frac{U}{\Uinf}}{l}^{y_j} = \prod_{i=1}^m\avo{\frac{U_i}{U_{i,\infty}}}{l}^{y_{j,i}} = \prod_{i=1}^m\frac{\avo{U_i}{l}^{y_{j,i}}}{U_{i,\infty}^{y_{j,i}}}.
		\end{equation*}
	\end{lemma}
	\begin{proof}
		By Poincar\'e-Wirtinger inequality \blue{and assumption \eqref{B}}, we have
		\begin{equation*}
			\int_{\Omega}\frac{|\na u_i|^2}{u_i}dx = 4\intO |\na U_i|^2dx \gtrsim \int_{\omega_l}|\na U_i|^2dx \gtrsim \|U_i - \avo{U_i}{l}\|_{L^2(\omega_l)}^2.
		\end{equation*}
		By using the elementary inequality $\Psi(x;y) = x\log(x/y) - x + y \ge (\sqrt x - \sqrt y)^2$ we have
		\begin{equation}\label{f0}
			\begin{aligned}
				\sum_{i=1}^m\|U_i - \avo{U_i}{l}\|_{L^2(\omega_l)}^2 +  \sum_{j=L_{l-1}+1}^{L_l}\int_{\omega_l}\uinf^{y_{j}}\Psi\bra{\frac{u^{y_j}}{\uinf^{y_j}}; \frac{u^{y_j'}}{\uinf^{y_j'}}}dx\\
				\gtrsim\sum_{i=1}^m\|U_i - \avo{U_i}{l}\|_{L^2(\omega_l)}^2 + \sum_{j=L_{l-1}+1}^{L_l}\int_{\omega_l} \left(\frac{U^{y_j}}{\Uinf^{y_j}} - \frac{U^{y_j'}}{\Uinf^{y_j'}}\right)^2dx,
			\end{aligned}
		\end{equation}	
		\blue{where we used the strict positivity of $\uinf$.}
		For any $i = L_{l-1}+1, \ldots, L_l$, we use the notation
		\begin{equation*}
			\eta_i(x):= U_i(x) - \avo{U_{i}}{l}, \; x\in\omega_l.
		\end{equation*}
		Fix a constant $\mathfrak{m}>0$, we consider the domain decomposition
		\begin{equation*}
			\blue{\omega_l = \Upsilon_l\cap \Upsilon^{c}_l}
		\end{equation*}
		where \blue{$\Upsilon_{l}:= \{x\in\omega_l: |\eta_i(x)| \le \mathfrak{m} \text{ for all } i\in \{1,\ldots, m\}\}$},
		and \blue{$\Upsilon^{c}_l = \omega_l\backslash\Upsilon_l$}. By Taylor's expansion
		\begin{equation*}
			U_i(x)^{y_{j,i}} = \bra{\avo{U_i}{l} + \eta_i(x)}^{y_{j,i}} = \avo{U_i}{l}^{y_{j,i}} + \wt{R}_i\eta_i(x),
		\end{equation*}
		with
		\begin{equation*}
			\wt{R}_i(x) = y_{j,i}(\theta\avo{U_i}{l} + (1-\theta)\eta_i(x))^{y_{j,i}-1} \text{ for some } \blue{\theta = \theta(i,j,l,x)\in (0,1)}.
		\end{equation*}
		Thanks to Lemma \ref{lem1} and the definition of $\Upsilon_l$, it holds that 
		\begin{equation}\label{f1}
			|\wt{R}_i(x)| \lesssim_{\mathfrak{m}} 1, \quad \forall x\in \Upsilon_l.
		\end{equation}
		Therefore, using $|\eta_i(x)| \le \mathfrak{m}$ in $\Upsilon_l$, \eqref{f1}, and the elementary inequality $(x-y)^2 \ge x^2/2 - y^2$, we can estimate
		\begin{equation}\label{f2}
			\begin{aligned}
				&\sum_{j=L_{l-1}+1}^{L_l}\int_{\Upsilon_l} \left(\frac{U^{y_j}}{\Uinf^{y_j}} - \frac{U^{y_j'}}{\Uinf^{y_j'}}\right)^2dx\\
				&= \sum_{j=L_{l-1}+1}^{L_l}\int_{\Upsilon_l}\bra{\prod_{i=1}^{m}\frac{\avo{U_i}{l}^{y_{j,i}} + \wt{R}_i\eta_i(x)}{U_{i,\infty}^{y_{j,i}}} - \prod_{i=1}^{m}\frac{\avo{U_i}{l}^{y_{j,i}'} + \wt{R}_i\eta_i(x)}{U_{i,\infty}^{y_{j,i}'}}}^2   \, dx\\
				&\ge \frac 12|\Upsilon_l| \sum_{j=L_{l-1}+1}^{L_l}\bra{\avo{\frac{U}{U_\infty}}{l}^{y_j} - \avo{\frac{U}{U_\infty}}{l}^{y_j'}}^2 - \mathcal{C}\sum_{i=1}^{m}\int_{\Upsilon_l}|\eta_i(x)|^2dx
			\end{aligned}
		\end{equation}
		where $\mathcal{C} = \mathcal{C}(\mathfrak{m})$. On the other hand, on $\Upsilon^c_l$ we know that there exists \blue{$i_0\in \{1,\ldots, m\}$} such that $|\eta_{i_0}(x)|\ge \mathfrak{m}$. Thus, we estimate (using first Cauchy-Schwarz inequality, and then Lemma \ref{lem1}) 
		\begin{equation}\label{f3}
			\begin{aligned}
				\sum_{i=1}^{m}\|U_i - \avo{U_i}{l}\|_{L^2(\omega_l)}^2 &= \int_{\omega_l}\sum_{i=1}^m|\eta_i(x)|^2dx\\
				&\ge \frac{1}{m}\int_{\Upsilon^c_l}\bra{\sum_{i=1}^m|\eta_i(x)|}^2dx\\
				&\ge \frac{\mathfrak{m}^2}{m}|\Upsilon^c_l|			 \gtrsim_{u_\infty,\mathfrak{m}}|\Upsilon^c_l|\sum_{j=L_{l-1}+1}^{L_l}\bra{\avo{\frac{U}{U_\infty}}{l}^{y_j} - \avo{\frac{U}{U_\infty}}{l}^{y_j'}}^2 .
			\end{aligned}
		\end{equation}
		From \eqref{f1} and \eqref{f2}, we can estimate for any $\delta\in (0,1)$, recalling that $\eta_i(x):= U_i(x) - \avo{U_{i}}{l}$,
		\begin{equation}\label{f4}
			\begin{aligned}
				\text{RHS of (\ref{f0})}&\ge \frac 12\sum_{i=1}^m\|\eta_i\|_{L^2(\omega_l)}^2 + \frac 12\sum_{i=1}^m\|\eta_i\|_{L^2(\omega_l)}^2 + \delta\sum_{j=L_{l-1}+1}^{L_l}\int_{\Upsilon_l} \left(\frac{U^{y_j}}{\Uinf^{y_j}} - \frac{U^{y_j'}}{\Uinf^{y_j'}}\right)^2dx\\
				& \gtrsim_{\mathfrak{m}  } \frac 12\sum_{i=1}^m\|\eta_i\|_{L^2(\omega_l)}^2 +  \bra{|\Upsilon^c_l| + \frac{\delta}{2}|\Upsilon_l|}\sum_{j=L_{l-1}+1}^{L_l}\bra{\avo{\frac{U}{U_\infty}}{l}^{y_j} - \avo{\frac{U}{U_\infty}}{l}^{y_j'}}^2\\
				&\qquad - \delta \mathcal{C}\sum_{i=1}^m\int_{\Upsilon_l}|\eta_i|^2dx\\
				&\gtrsim_{\mathfrak{m} }|\omega_l|\sum_{j=L_{l-1}+1}^{L_l}\bra{\avo{\frac{U}{U_\infty}}{l}^{y_j} - \avo{\frac{U}{U_\infty}}{l}^{y_j'}}^2,
			\end{aligned}
		\end{equation}
		where we choose $\delta$ small enough depending on $\mathcal{C}$, and consequently on $\mathfrak{m}$. It is remarked that the last inequality is {\it not dependent on $\Upsilon_l$}. This last estimate and \eqref{f0} allow to conclude the proof of Lemma \ref{lem2}.
	\end{proof}
	\begin{remark}
		Clearly $\mathfrak{m}$ can be arbitrary in the proof of Lemma \ref{lem2}, and one can choose, for instance, $\mathfrak{m} = 1$. We chose to write $\mathfrak{m}$ as a constant to leave the room for optimising constants in the desired inequality.
	\end{remark}
	\begin{lemma}\label{lem3}
		For each $l \in \{1,\ldots, s\}$, it holds
		\begin{equation*}
			\sum_{j=L_{l-1}+1}^{L_l} \left(\avo{\frac{U}{\Uinf}}{l}^{y_j} - \avo{\frac{U}{\Uinf}}{l}^{y_j'}\right)^2 \gtrsim \sum_{j=L_{l-1}+1}^{L_l} \left(\avO{\frac{U}{\Uinf}}^{y_j} - \avO{\frac{U}{\Uinf}}^{y_j'}\right)^2 - \sum_{i=1}^m\abs{\avO{U_i} - \avo{U_i}{l}}^2.
		\end{equation*}
	\end{lemma}
	\begin{proof}
		Denote by $\gamma_{i,l} = \avO{U_i} - \avo{U_i}{l}$. It follows from Lemma \ref{lem1} that
		\begin{equation*}
			|\gamma_{i,l}| \lesssim 1 \quad \forall i=1,\ldots, m, \quad \forall l=1,\ldots, s.
		\end{equation*}
		By Taylor's expansion,
		\begin{align*}
			\avo{\frac{U}{\Uinf}}{l}^{y_j} - \avo{\frac{U}{\Uinf}}{l}^{y_j'} &= \prod_{i=1}^{m}\frac{\bra{\avO{U_i} - \gamma_{i,l}}^{y_{j,i}}}{U_{i,\infty}^{y_{j,i}}} - \prod_{i=1}^{m}\frac{\bra{\avO{U_i} - \gamma_{i,l}}^{y_{j,i}'}}{U_{i,\infty}^{y_{j,i}'}}\\
			&= \prod_{i=1}^{m}\frac{\avO{U_i}^{y_{j,i}} - \gamma_{i,l}\RR(\avO{U_i},\gamma_{i,l},y_j)}{U_{i,\infty}^{y_{j,i}}} - \prod_{i=1}^{m}\frac{\avO{U_i}^{y_{j,i}'} - \gamma_{i,l}\RR(\avO{U_i},\gamma_{i,l},y_j')}{U_{i,\infty}^{y_{j,i}'}}\\
			&= \avO{\frac{U}{\Uinf}}^{y_j} - \avO{\frac{U}{\Uinf}}^{y_j'} - \wt{\RR}(\avO{U}, \gamma_{i,l}, y_j,y_j',\Uinf) \sum_{i=1}^m|\gamma_{i,l}| ,
		\end{align*}
		where \blue{$\RR(\cdot)$ denote the rest terms from Taylor expansions and $\tilde \RR$ is computed from $\RR$}. Thanks to Lemma \ref{lem1} and the bounds of $\gamma_{i,l}$,
		\begin{equation*}
			|\RR(\avO{U_i},\gamma_{i,l},y_j)| + |\RR(\avO{U_i},\gamma_{i,l},y_j')| +  |\wt{\RR}(\avO{U}, \gamma_{i,l}, y_j,y_j',\Uinf)| \lesssim 1.
		\end{equation*}
		Therefore, by using the elementary inequality $(x + y)^2 \ge \frac 12x^2 - y^2$, we have
		\begin{align*}
			\bra{\avo{\frac{U}{\Uinf}}{l}^{y_j} - \avo{\frac{U}{\Uinf}}{l}^{y_j'}}^2 &\gtrsim \frac 12 \bra{\avO{\frac{U}{\Uinf}}^{y_j} - \avO{\frac{U}{\Uinf}}^{y_j'}}^2 - \sum_{i=1}^m|\gamma_{i,l}|^2 .
		\end{align*}
		By summing for $j=L_{l-1}+1, \ldots, L_{l}$, we can conclude the proof of Lemma \ref{lem3}.
	\end{proof}
	\begin{lemma}\label{lem4}
		It holds that
		\begin{equation*}
			\sum_{i=1}^m\intO \frac{|\nabla u_i|^2}{u_i}dx \gtrsim \sum_{l=1}^{s}\sum_{i=1}^{m}\abs{\avO{U_i} - \avo{U_i}{l}}^2.
		\end{equation*}
	\end{lemma}
	\begin{proof}
		For any $i\in \{1,\ldots m\}$ and any $l\in \{1,\ldots, s\}$, we use Poincar\'e-Wirtinger inequality to estimate
		\begin{equation*}
			\intO\frac{|\na u_i|^2}{u_i}dx = 4\|\na U_i\|_{L^2(\Omega)}^2 \gtrsim \|U_i - \avO{U_i}\|_{L^2(\Omega)}^2 \ge \|U_i - \avO{U_i}\|_{L^2(\omega_l)}^2
		\end{equation*}
		and
		\begin{equation*}
			\intO\frac{|\na u_i|^2}{u_i}dx = 4\|\na U_i\|_{L^2(\Omega)}^2 \gtrsim \|\na U_i\|_{L^2(\omega_l)}^2 \gtrsim \|U_i - \avo{U_i}{l}\|_{L^2(\omega_l)}^2.
		\end{equation*}
		Therefore,
		\begin{equation*}
			\intO\frac{|\na u_i|^2}{u_i}dx \gtrsim \int_{\omega_l}\bra{\abs{U_i - \avO{U_i}}^2 + \abs{U_i - \avo{U_i}{l}}^2}dx \gtrsim |\omega_l|\abs{\avO{U_i} - \avo{U_i}{l}}^2
		\end{equation*}
		which concludes the proof of Lemma \ref{lem4}.
	\end{proof}
	\begin{lemma}\label{lem5}
		It holds that
		\begin{equation*}
			\D(u) \gtrsim \sum_{r=1}^{R}\bra{\avO{\frac{U}{\Uinf}}^{y_r} - \avO{\frac{U}{\Uinf}}^{y_r'}}^2.
		\end{equation*}
	\end{lemma}
	\begin{proof}
		Let $\theta\in (0,1)$ to be chosen later. It follows from \eqref{h1}, Lemmas \ref{lem2} and \ref{lem3} that
		\begin{align*}
			\theta \D(u) \gtrsim \theta \sum_{r=1}^{R}\bra{\avO{\frac{U}{\Uinf}}^{y_r} - \avO{\frac{U}{\Uinf}}^{y_r'}}^2 - \theta\sum_{l=1}^{s}\sum_{i=1}^{m}\abs{\avO{U_i} - \avo{U_i}{l}}^2.
		\end{align*}
		On the other hand, by Lemma \ref{lem4},
		\begin{equation*}
			(1-\theta)\D(u) \gtrsim (1-\theta)\sum_{l=1}^{s}\sum_{i=1}^{m}\abs{\avO{U_i} - \avo{U_i}{l}}^2.
		\end{equation*}
		Thus, by choosing $\theta$ small enough we get the desired estimate in Lemma \ref{lem5}.
	\end{proof}
	
	\medskip
	\noindent{\bf Proof of Proposition \ref{pro:funtional_inequality}.}
	Proposition \ref{pro:funtional_inequality} now follows immediately from Lemma \ref{lem5} and \cite[Lemmas 2.7 and 2.8]{fellner2018convergence}. For the convenience of the reader, we nevertheless provide the proof. First, from \eqref{h1} and Lemma \ref{lem5}, we have
	\begin{equation}\label{h2}
		\D(u) \gtrsim \sum_{i=1}^m\int_{\Omega}\frac{|\na u_i|^2}{u_i}dx + \sum_{r=1}^{R}\bra{\avO{\frac{U}{\Uinf}}^{y_r} - \avO{\frac{U}{\Uinf}}^{y_r'}}^2.
	\end{equation}
	Thanks to the Logarithmic Sobolev inequality in bounded Lipschitz domains, see e.g. \cite{desvillettes2014exponential}, we have
	\begin{equation}\label{h3}
		\D(u) \gtrsim \sum_{i=1}^{m}\int_{\Omega}u_i\log\frac{u_i}{\avO{u_i}}dx = \E(u|\avO{u}),
	\end{equation}
	with $\E(u|\avO{u})$ appearing in \eqref{entropy_additive}. From Poincar\'e-Wirtinger's inequality 
	\begin{equation}\label{h3_1}
		\int_{\Omega}\frac{|\na u_i|^2}{u_i}dx = 4\int_{\Omega}|\na U_i|^2dx \gtrsim_{\Omega} \|U_i - \avO{U_i}\|_{\LO{2}}^2 =: \|\mu_i\|_{\LO{2}}^2 ,
	\end{equation}
	where we denote by $\mu_i(x):= U_i(x) - \avO{U_i}$ for $x\in\Omega$, $i=1,\ldots, m$. Thanks to Lemma \ref{lem1},
	\begin{equation}\label{h4}
		\|\mu_i\|_{\LO{2}} \lesssim 1.
	\end{equation}
	From $\|\mu_i\|_{\LO{2}}^2 = \avO{U_i^2} - \avO{U_i}^2$, 
	\begin{equation*}
		\frac{\avO{U_i}}{U_{i,\infty}} = \frac{1}{U_{i,\infty}}\bra{ \sqrt{ \avO{U_i^2} } - \frac{\|\mu_i\|_{\LO{2}}^2}{\sqrt{\avO{U_i^2}}+\avO{U_i}}} = \sqrt{\frac{\avO{u_i}}{u_{i,\infty}}} - \mathfrak{R}(U_i)\|\mu_i\|_{L^2(\Omega)},
	\end{equation*}
	with
	\begin{equation*}
		\mathfrak{R}(U_i):= \frac{\|\mu_i\|_{\LO{2}}}{U_{i,\infty}\bra{\sqrt{\avO{U_i^2}} + \avO{U_i}}} \ge 0
	\end{equation*}
	satisfying
	\begin{equation}\label{h5}
		\mathfrak{R}(U_i)^2 = \frac{\|\mu_i\|_{\LO{2}}^2}{U_{i,\infty}^2\bra{\sqrt{\avO{U_i^2}} + \avO{U_i}}^2} = \frac{\sqrt{\avO{U_i^2}} - \avO{U_i}}{{U_{i,\infty}^2\bra{\sqrt{\avO{U_i^2}} + \avO{U_i}}}} \le \frac{1}{U_{i,\infty}^2}.
	\end{equation}
	Therefore, we can estimate using Taylor's expansion and the bounds \eqref{h4}, \eqref{h5}
	\begin{equation}\label{h6}
		\begin{aligned}
			&\sum_{r=1}^{R}\bra{\avO{\frac{U}{\Uinf}}^{y_r} - \avO{\frac{U}{\Uinf}}^{y_r'}}^2\\
			&= \sum_{r=1}^{R}\left[\prod_{i=1}^m\bra{\sqrt{\frac{\avO{u_i}}{u_{i,\infty}}} - \mathfrak{R}(U_i)\|\mu_i\|_{L^2(\Omega)}}^{y_{r,i}} - \prod_{i=1}^m\bra{\sqrt{\frac{\avO{u_i}}{u_{i,\infty}}} - \mathfrak{R}(U_i)\|\mu_i\|_{L^2(\Omega)}}^{y_{r,i}'} \right]^2\\
			&\gtrsim_{u_\infty} \frac{1}{2}\sum_{r=1}^{R}\left[\sqrt{\frac{\avO{u_i}}{u_{i,\infty}}}^{y_r} - \sqrt{\frac{\avO{u_i}}{u_{i,\infty}}}^{y_r'} \right]^2 - \sum_{i=1}^m\|\mu_i\|_{\LO{2}}^2.
		\end{aligned}
	\end{equation}
	Now, it follows from \eqref{h2}, \eqref{h3_1} and \eqref{h6} that for any $\theta\in (0,1)$
	\begin{equation}
		\begin{aligned}
			\D(u) &\gtrsim \sum_{i=1}^m\int_{\Omega}\frac{|\na u_i|^2}{u_i}dx + \theta \sum_{r=1}^{R}\bra{\avO{\frac{U}{\Uinf}}^{y_r} - \avO{\frac{U}{\Uinf}}^{y_r'}}^2\\
			&\gtrsim \sum_{i=1}^m\|\mu_i\|_{\LO{2}}^2 + \frac{\theta}{2}\sum_{r=1}^{R}\left[\sqrt{\frac{\avO{u_i}}{u_{i,\infty}}}^{y_r} - \sqrt{\frac{\avO{u_i}}{u_{i,\infty}}}^{y_r'} \right]^2 - \theta \sum_{i=1}^m\|\mu_i\|_{\LO{2}}^2\\
			&\gtrsim_{\theta}\sum_{r=1}^{R}\left[\sqrt{\frac{\avO{u_i}}{u_{i,\infty}}}^{y_r} - \sqrt{\frac{\avO{u_i}}{u_{i,\infty}}}^{y_r'} \right]^2,
		\end{aligned}
		\label{h7}
	\end{equation}
	by choosing $\theta$ small enough. Applying then \cite[Inequality (11)]{fellner2018convergence}, we have the following finite dimensional inequality
	\begin{equation}\label{h8}
		\sum_{r=1}^{R}\left[\sqrt{\frac{\avO{u_i}}{u_{i,\infty}}}^{y_r} - \sqrt{\frac{\avO{u_i}}{u_{i,\infty}}}^{y_r'} \right]^2 \gtrsim \sum_{i=1}^m\bra{\sqrt{\frac{\avO{u_i}}{u_{i,\infty}}} - 1}^2.
	\end{equation}
	On the other hand, using the elementary inequality $z\log z - z + 1 \le (z - 1)^2$, we estimate
	\begin{equation}\label{h9}
		\begin{aligned}
			\E(\avO{u}|u_\infty) &= \sum_{i=1}^mu_{i,\infty}\bra{ \frac{\avO{u_i}}{u_{i,\infty}} \log \frac{\avO{u_i}}{u_{i,\infty}} - \frac{\avO{u_i}}{u_{i,\infty}} + 1}\\
			&\le \sum_{i=1}^mu_{i,\infty}\bra{\frac{\avO{u_i}}{u_{i,\infty}}-1}^2\\
			&= \sum_{i=1}^mu_{i,\infty}\bra{\sqrt{\frac{\avO{u_i}}{u_{i,\infty}}}+1}^2\bra{\sqrt{\frac{\avO{u_i}}{u_{i,\infty}}}-1}^2\\
			&\lesssim \sum_{i=1}^m\bra{\sqrt{\frac{\avO{u_i}}{u_{i,\infty}}}-1}^2.
		\end{aligned}
	\end{equation}
	Combining \eqref{h7}, \eqref{h8} and \eqref{h9} yields
	\begin{equation*}
		\D(u) \gtrsim \E(\avO{u}|u_\infty).
	\end{equation*}
	From this, \eqref{h3} and \eqref{entropy_additive}, we get the proof of Proposition \ref{pro:funtional_inequality}.
	
	\hfill $\square$
	
	\medskip
	We are now ready to prove the first main result.
	
	\noindent{\bf Proof of Theorem \ref{thm:main1}.}
	Thanks to Lemma \ref{lem0}, any renormalised solution satisfies for $0\le \tau<T$
	\begin{equation}\label{e1}
		\E(u(s)|\uinf)\bigg|_{s=\tau}^{s=T} + \int_{\tau}^{T} \D(u(s))ds \le 0.
	\end{equation}
	Moreover, still thanks to Lemma \ref{lem0}, all renormalised solutions satisfy the conservation laws \eqref{conservation_laws}. Note that \eqref{e1} also implies 
	\begin{equation*} 
		\sum_{i=1}^m\int_{\Omega}u_i(x,t)\log u_i(x,t)dx \le L
	\end{equation*}
	for some $L$ depending on $\E(u_0|u_\infty)$ and $u_\infty$. Therefore, we can apply Proposition \ref{pro:funtional_inequality} to get
	\begin{equation*}
		\D(u(s)) \ge \lambda \E(u(s)|\uinf),\quad \forall s\ge 0.
	\end{equation*}
	Inserting this into \eqref{e1} gives for all $0\le \tau < T$
	\begin{equation*}
		\E(u(s)|\uinf)\bigg|_{s=\tau}^{s=T} + \lambda\int_{\tau}^{T} \E(u(s)|\uinf)ds \le 0.
	\end{equation*}
	Defining
	\begin{equation*}
		\varphi(\tau) = \int_{\tau}^{T}\E(u(s)|\uinf)ds ,
	\end{equation*}		
	we see that
	\begin{equation*}
		\varphi'(\tau) = -\E(u(\tau)|\uinf) \le -\E(u(T)|\uinf) -\lambda\int_{\tau}^{T}\E(u(s)|\uinf)ds = -\E(u(T)|\uinf) -\lambda \varphi(\tau).
	\end{equation*}
	Gronwall's lemma implies then that
	\begin{equation*}
		e^{\lambda T}\varphi(T) + \E(u(T)|\uinf)\frac{e^{\lambda T}-e^{\lambda\tau}}{\lambda} \le e^{\lambda \tau}\varphi(\tau).
	\end{equation*}
	Using $\varphi(T)=0$ and $\lambda\varphi(\tau)\le \E(u(\tau)|\uinf) - \E(u(T)|\uinf)$ leads to
	\begin{equation*}
		\E(u(T)|\uinf)(e^{\lambda T} - e^{\lambda\tau}) \le e^{\lambda \tau}(\E(u(\tau)|\uinf) - \E(u(T)|\uinf),
	\end{equation*}
	and thus to
	\begin{equation*}
		\E(u(T)|\uinf) \le e^{-\lambda(T-\tau)}\E(u(\tau)|\uinf) \quad \forall 0\le \tau < T.
	\end{equation*}
	Setting $\tau = 0$ entails the global exponential decay of the relative entropy, which finally yields the decay of the solution towards equilibrium, thanks to Lemma \ref{lem:CSK}.
	\hfill $\square$
	
	\section{Reactions in measurable sets - Proof of Theorem \ref{thm:measurable}}\label{sec3}
	
	\subsection{Preliminary estimates}
	\begin{proposition}\label{pro1}
		Assume the assumptions in Theorem \ref{thm:measurable}. Then, for any bounded and non-negative initial data $u_0 \in L^{\infty}_+(\Omega)^3$, there exists a unique global non-negative weak solution to \eqref{special_sys}, which is bounded uniformly in time, i.e.
		\begin{equation}\label{Linf-bound}
			\sup_{t\ge 0}\sup_{i=1,2,3}\|u_i(t)\|_{L^{\infty}(\Omega)} < +\infty.
		\end{equation}
	\end{proposition}
	
	
	
	\begin{proof}
		Denote by $f_j(x,t,u)$ the nonlinearity in the equation for $u_j$ in \eqref{special_sys}.
		It is easy to check that these nonlinearities are locally Lipschitz continuous in $u$, uniformly in $(x,t)$, quasi-positive, i.e.
		\begin{equation*} 
			f_j(x,t,u)\ge 0 \quad \text{ for all } (x,t,u)\in \Omega\times\mathbb R_+\times\mathbb R_+^3 \text{ with } u_j = 0,
		\end{equation*}
		and satisfy the following (weighted) conservation of mass condition
		\begin{equation*} 
			4f_1(x,t,u) + 2f_2(x,t,u) + f_3(x,t,u) = 0, \quad \forall (x,t,u)\in \Omega\times\mathbb R_+\times\mathbb R_+^3.
		\end{equation*}
		Therefore, we can apply \cite[Theorem 1.1]{fitzgibbon2021reaction} to obtain the global existence of a unique weak solution, together with the uniform-in-time boundedness in $\LO{\infty}$-norm.
	\end{proof}
	
	In the following, we prove certain estimates linked to the entropy and entropy dissipation of \eqref{special_sys}.
	
	\medskip
	Thanks to the $\LO{\infty}$ bound \eqref{Linf-bound} in Proposition \ref{pro1}, we denote
	\begin{equation}\label{C0}
		C_{0} :=\sup_{t\ge 0}(\|u_1(t)\|_{\LO{\infty}} + \|u_2(t)\|_{\LO{\infty}} + \|u_3(t)\|_{\LO{\infty}}).
	\end{equation}
	As in the previous section, we consider the relative entropy
	\begin{equation}\label{entropy_special}
		\mathcal{E}(u|u_\infty) =\sum_{j=1}^{3}\int_{\Omega}  \bigg( u_{j}\left(
		\cdot,t\right)  \ln\frac{u_{j}\left(  \cdot,t\right)  }{u_{j,\infty}%
		}-u_{j}\left(  \cdot,t\right)  +u_{j,\infty}  \bigg) \, dx,
	\end{equation}
	where the equilibrium $u_{\infty}$ is defined as in \eqref{equi_sys_special}, and the corresponding entropy dissipation
	\begin{equation}\label{entropy_dissipation_special}
		\mathcal{D}(u) := \sum_{j=1}^3\int_{\Omega}d_{j}\frac{\left\vert \nabla
			u_{j}\right\vert ^{2}}{u_{j}}dx+\int_{\Omega}k_{1}\left(  u_{2}^{2}%
		-u_{1}\right)  \ln\frac{u_{2}^{2}}{u_{1}}dx+\int_{\Omega}k_{2}\left(
		u_{3}^{2}-u_{2}\right)  \ln\frac{u_{3}^{2}}{u_{2}}dx .
	\end{equation}
	
	\begin{lemma}\label{lem6}
		It holds
		\begin{equation*}
			\sum_{j=1}^3\|u_j-u_{j,\infty}\|_{L^1(\Omega)}^2 \lesssim_{u_\infty,C_0} \E(u|u_\infty) \lesssim_{u_\infty} \sum_{j=1}^3\int_{\Omega}|u_j - u_{j,\infty}|^2dx.
		\end{equation*}
	\end{lemma}
	\begin{proof}
		The first estimate is a special case of the Csisz\'ar-Kullback-Pinsker inequality in Lemma \ref{lem:CSK},
		and the second one follows directly form the elementary inequality
		\begin{equation*}
			x\log\frac{x}{y} - x + y \le \frac{1}{y}|x - y|^2.
		\end{equation*}
	\end{proof}
	
	\begin{lemma}\label{lem7}
		For solutions to \eqref{special_sys}, it holds, for $j=1,2,3$,
		\begin{equation}\label{e2}
			\norm{\sqrt{u_j} - \avO{\sqrt{u_j}}}_{L^2(\Omega)}^2 \lesssim_{\Omega}\intO \frac{|\na u_j|^2}{u_j}dx,
		\end{equation}
		and
		\begin{equation}\label{e3}
			\norm{u_j - \avO{u_j}}_{L^2(\Omega)}^2 \lesssim_{\Omega,C_0} \intO \frac{|\na u_j|^2}{u_j}dx,
		\end{equation}
		with $C_0$ defined in \eqref{C0}.
	\end{lemma}
	\begin{proof}
		The estimate \eqref{e2} is a consequence of Poincar\'e-Wirtinger's inequality and the fact that $\frac{|\na u_j|^2}{u_j} = 4|\na \sqrt{u_j}|^2$. For \eqref{e3}, we estimate
		\begin{equation*}
			\intO|u_j - \avO{u_j}|^2dx = \intO\abs{\sqrt{u_j} + \sqrt{\avO{u_j}}}^2\abs{\sqrt{u_j} - \avO{\sqrt{u_j}}}^2dx \lesssim_{C_0}\norm{\sqrt{u_j} - \avO{\sqrt{u_j}}}_{L^2(\Omega)}^2,
		\end{equation*}
		hence \eqref{e3} follows from \eqref{e2}.
	\end{proof}
	
	\begin{lemma}\label{lem7_1}
		It holds
		\begin{equation*}
			\D(u) \gtrsim \sum_{j=1}^3\int_{\Omega}d_j\frac{|\na u_j|^2}{u_j}dx + \int_{\omega_1}\abs{u_2 - \sqrt{u_1}}^2dx + \int_{\omega_2}\abs{u_3 - \sqrt{u_2}}^2dx.
		\end{equation*}
	\end{lemma}
	\begin{proof}
		The proof is straightforward thanks to the inequality $(x-y)\log(x/y) \ge (\sqrt x - \sqrt y)^2$, and assumptions \eqref{k1k2}.
	\end{proof}
	From the equations defining the equilibrium and the conservation of total mass, that is
	\begin{equation}\label{e4}
		\left\{
		\begin{array}
			[c]{l}%
			u_{2,\infty}^{2}-u_{1,\infty}=0,\\
			u_{3,\infty}^{2}-u_{2,\infty}=0,\\
			4u_{1,\infty}+2u_{2,\infty}+u_{3,\infty}=\displaystyle\int_{\Omega}\left(
			4u_{1}+2u_{2}+u_{3}\right)dx  \text{,}%
		\end{array}
		\right.
	\end{equation}
	we have
	\begin{equation}\label{e5}
		\begin{aligned}
			4u_{1,\infty} + 2\sqrt{u_{1,\infty}} + \sqrt[4]{u_{1,\infty}} &= 4u_{2,\infty}^2 + 2u_{2,\infty} + \sqrt{u_{2,\infty}}\\
			&= 4u_{3,\infty}^4 + 2u_{3,\infty}^2 + u_{3,\infty}\\
			&= \int_{\Omega}(4u_1 + 2u_2 + u_3)dx.
		\end{aligned}
	\end{equation}
	
	\begin{lemma}\label{lem8}
		We have the following pointwise estimates for all $(x,t)\in \Omega\times\R_+$
		\begin{equation*}
			|u_1 - u_{1,\infty}|^2 \lesssim_{C_0,u_\infty} |u_2 - \sqrt{u_1}|^2 + |u_2 - u_{2,\infty}|^2,
		\end{equation*}
		\begin{equation*}
			|u_3 - u_{3,\infty}|^2 \lesssim_{u_\infty} |u_3 - \sqrt{u_2}|^2 + |u_2 - u_{2,\infty}|^2,
		\end{equation*}
		and
		\begin{equation}\label{e5_1}
			|u_2 - u_{2,\infty}|^2 \lesssim_{C_0} \sum_{j=1}^3|u_j - \avO{u_j}|^2 + |u_2 - \sqrt{u_1}|^2 + |u_3 - \sqrt{u_2}|^2.
		\end{equation}
	\end{lemma}
	\begin{proof}
		By using \eqref{C0} and elementary computations, we have
		\begin{align*}
			|u_1 - u_{1,\infty}|^2 = |u_1 - u_{2,\infty}^2|^2
			&\lesssim |u_1 - u_2^2|^2 + |u_2^2 - u_{2,\infty}^2|^2\\
			&\lesssim_{C_0,u_\infty} |u_2 - \sqrt{u_1}|^2 + |u_2 - u_{2,\infty}|^2,
		\end{align*}
		and similarly,
		\begin{align*}
			|u_3 - u_{3,\infty}|^2 = |u_3 - \sqrt{u_{2,\infty}}|^2
			&\lesssim |u_3 - \sqrt{u_2}|^2 + |\sqrt{u_2} - \sqrt{u_{2,\infty}}|^2\\
			&\lesssim_{u_\infty}|u_3 - \sqrt{u_2}|^2 + |u_2 - u_{2,\infty}|^2.
		\end{align*}
		Defining $f(z) := 4z^2 + 2z + \sqrt{z}$, $z> 0$ we see that
		\begin{equation*}
			|f(w) - f(z)| = |w-z|\abs{4(w+z) + 2 + \frac{1}{\sqrt w + \sqrt z}} \ge 2|w-z|.
		\end{equation*}
		Therefore,
		\begin{align*}
			|u_2 - u_{2,\infty}|^2 &\le |f(u_2) - f(u_{2,\infty})|^2\\
			&= \abs{4u_2^2 + 2u_2 + \sqrt{u_2} - \intO(4u_1 + 2u_2 + u_3)dx}^2 \quad (\text{using }(\ref{e5}))\\
			&\lesssim |u_2^2 - u_1|^2 + |u_1 - \avO{u_1}|^2 + |u_2 - \avO{u_2}|^2 + |\sqrt{u_2} - u_3|^2 + |u_3 - \avO{u_3}|^2\\
			&\lesssim_{C_0}   \sum_{j=1}^3|u_j - \avO{u_j}|^2  + |u_2 - \sqrt{u_1}|^2 + |u_3 - \sqrt{u_2}|^2.
		\end{align*}
	\end{proof}

	\subsection{Proof of Theorem \ref{thm:measurable}} 
	
	\begin{lemma}\label{lem10}
		Under the assumptions of Theorem \ref{thm:measurable}, it holds
		\begin{equation*} 
			\E(u|u_\infty) \lesssim \frac{1}{|\omega_1|} \sum_{j=1}^3\|u_j- \avO{u_{j}}\|_{\LO{2}}^2 + \frac{1}{|\omega_1|} \int_{\omega_1}\bra{|u_2 - \sqrt{u_1}|^2 + |u_3 - \sqrt{u_2}|^2}dx.
		\end{equation*}
	\end{lemma}
	\begin{proof}
		First, we have
		\begin{align*}
			\intO|u_1 - u_{1,\infty}|^2dx &\lesssim \|u_1 - \avO{u_1}\|_{L^2(\Omega)}^2 + |\avO{u_1} - u_{1,\infty}|^2.
		\end{align*}
		Then, it holds
		\begin{align*}
			|\avO{u_1} - u_{1,\infty}|^2 &= \frac{1}{|\omega_1|}\int_{\omega_1}|\avO{u_1} - u_{1,\infty}|^2dx\\
			&\lesssim \frac{1}{|\omega_1|}\int_{\omega_1}|u_1 - \avO{u_1}|^2dx + \frac{1}{|\omega_1|}\int_{\omega_1}|u_1 - u_{1,\infty}|^2dx\\
			&\lesssim \frac{1}{|\omega_1|}\|u_1 - \avO{u_1}\|_{\LO{2}}^2 + \frac{1}{|\omega_1|}\int_{\omega_1}\bra{|u_2 - \sqrt{u_1}|^2 + |u_2 - u_{2,\infty}|^2}dx ,
		\end{align*}
		where we used Lemma \ref{lem8} at the last estimate.
		Therefore, we have
		\begin{equation}\label{e8} 
			\intO|u_1-u_{1,\infty}|^2dx \lesssim \frac{1}{|\omega_1|}\bra{\|u_1 - \avO{u_1}\|_{\LO{2}}^2 + \intO|u_2 - u_{2,\infty}|^2dx + \int_{\omega_1}|u_2 - \sqrt{u_1}|^2dx}.
		\end{equation}
		Similarly,
		\begin{equation}\label{e9}
			\intO|u_3-u_{3,\infty}|^2dx \lesssim \frac{1}{|\omega_1|}\bra{\|u_3 - \avO{u_3}\|_{\LO{2}}^2 + \intO|u_2 - u_{2,\infty}|^2dx + \int_{\omega_1}|u_3 - \sqrt{u_2}|^2dx}.
		\end{equation}
		Integrating both sides of \eqref{e5_1} over $\omega_1$ yields
		\begin{equation}\label{e9_1}
			\begin{aligned} 
				\int_{\omega_1}|u_2 - u_{2,\infty}|^2dx &\lesssim \sum_{j=1}^3\int_{\omega_1}|u_j - \avO{u_j}|^2dx + \int_{\omega_1}\bra{|u_2 - \sqrt{u_1}|^2 + |u_3 - \sqrt{u_2}|^2}dx\\
				&\lesssim\sum_{j=1}^3\|u_j - \avO{u_j}\|_{\LO{2}}^2 + \int_{\omega_1}\bra{|u_2 - \sqrt{u_1}|^2 + |u_3 - \sqrt{u_2}|^2}dx.
			\end{aligned}
		\end{equation}
		Thus
		\begin{equation}\label{e10}
			\begin{aligned}
				\int_{\Omega}|u_2 - u_{2,\infty}|^2dx &\lesssim \|u_2 - \avO{u_2}\|_{L^2(\Omega)}^2 + \frac{1}{|\omega_1|}\int_{\omega_1}|\avO{u_2} - u_{2,\infty}|^2dx\\
				&\lesssim \|u_2 - \avO{u_2}\|_{L^2(\Omega)}^2 + \frac{1}{|\omega_1|} \int_{\omega_1}|\avO{u_2} - u_2|^2dx + \frac{1}{|\omega_1|}\int_{\omega_1}|u_2 - u_{2,\infty}|^2dx\\
				&\lesssim \frac{1}{|\omega_1|}\sum_{j=1}^3\|u_j - \avO{u_j}\|_{\LO{2}}^2 + \frac{1}{|\omega_1|}\int_{\omega_1}\bra{|u_2 - \sqrt{u_1}|^2 + |u_3 - \sqrt{u_2}|^2}dx ,
			\end{aligned}
		\end{equation}
		where \eqref{e9_1} was used at the last step. Lemma \ref{lem10} then follows directly from Lemma \ref{lem6} and \eqref{e8}--\eqref{e10}.
	\end{proof}
	It is clear that we only have to control the term $\int_{\omega_1}|u_3 - \sqrt{u_2}|^2dx$ since the reaction $\S_2 \leftrightarrows 2\S_3$ happens in $\omega_2$ and not necessarily in $\omega_1$. This is done in the following lemma.
	
	\begin{lemma}\label{lem11}
		Under the assumptions of Theorem \ref{thm:measurable}, it holds 
		\begin{align*} 
			\frac{1}{|\omega_1|}\int_{\omega_1}|u_3 - \sqrt{u_2}|^2dx \lesssim_{C_0} \bra{\frac{1}{|\omega_1|}+\frac{1}{|\omega_2|}}\D(u).
		\end{align*}
	\end{lemma}
	\begin{proof}
		Using triangular inequality, we estimate
		\begin{equation}\label{ee}
			\begin{aligned}
				&\frac{1}{|\omega_1|}\int_{\omega_1}\abs{u_3 - \sqrt{u_2}}^2dx\\ &\lesssim \frac{1}{|\omega_1|}\int_{\omega_1}\bigg(\abs{u_3 - \avO{u_3}}^2 + 
				\abs{\avO{u_3} - \avo{u_3}{2}}^2+ \abs{\avo{u_3}{2} - \avo{\sqrt{u_2}}{2}}^2\\
				&\qquad  + \abs{\avo{\sqrt{u_2}}{2} - \avO{\sqrt{u_2}}}^2 +  \abs{ \avO{\sqrt{u_2}} - \sqrt{u_2}}^2\bigg)dx\\
				&\lesssim \frac{1}{|\omega_1|}\intO\bra{|u_3 - \avO{u_3}|^2 + |\sqrt{u_2} - \avO{\sqrt{u_2}}|^2}dx\\
				&\qquad+ \underbrace{\bra{\abs{\avO{u_3} - \avo{u_3}{2}}^2 + \abs{\avO{\sqrt{u_2}} - \avo{\sqrt{u_2}}{2}}^2}}_{(I)}+\underbrace{\abs{\avo{u_3}{2} - \avo{\sqrt{u_2}}{2}}^2}_{(II)}.
			\end{aligned}
		\end{equation}
		For $(I)$, we use Cauchy-Schwarz inequality to get
		\begin{align*}
			\abs{\avO{u_3} - \avo{u_3}{2}}^2 = \left|\frac{1}{|\omega_2|}\int_{\omega_{2}}(\avO{u_3}-u_3)dx \right|^2\le \frac{1}{|\omega_2|}\int_{\omega_2}|\avO{u_3}-u_3|^2dx \le \frac{1}{|\omega_2|}\int_{\Omega}|\avO{u_3}-u_3|^2dx ,
		\end{align*}
		and similarly
		\begin{equation*} 
			\abs{\avO{\sqrt{u_2}} - \avo{\sqrt{u_2}}{2}}^2 = \abs{\frac{1}{|\omega_2|}\int_{\omega_2}(\sqrt{u_2} - \avO{\sqrt{u_2}})dx}^2 \le \frac{1}{|\omega_2|}\intO \abs{\sqrt{u_2} - \avO{\sqrt{u_2}}}^2dx .
		\end{equation*}
		Concerning $(II)$, we estimate
		\begin{align*}
			(II) &= \abs{\frac{1}{|\omega_2|}\int_{\omega_2}(u_3 - \sqrt{u_2})dx}^2 \le \frac{1}{|\omega_2|}\int_{\omega_2}|u_3 - \sqrt{u_2}|^2dx.
		\end{align*}
		Therefore, it follows from \eqref{ee} that
		\begin{equation*} 
			\begin{aligned}
				&\frac{1}{|\omega_1|}\int_{\omega_1}\abs{u_3 - \sqrt{u_2}}^2dx\\
				&\lesssim \bra{\frac{1}{|\omega_1|}+\frac{1}{|\omega_2|}}\intO\bra{|u_3 - \avO{u_3}|^2 + |\sqrt{u_2} - \avO{\sqrt{u_2}}|^2}dx+\frac{1}{|\omega_2|}\int_{\omega_2}\abs{u_3 - \sqrt{u_2}}^2dx\\
				&\lesssim \bra{\frac{1}{|\omega_1|}+\frac{1}{|\omega_2|}}\bra{\intO\bra{|u_3 - \avO{u_3}|^2 + |\sqrt{u_2} - \avO{\sqrt{u_2}}|^2}dx + \int_{\omega_2}\abs{u_3 - \sqrt{u_2}}^2dx}\\
				&\lesssim_{\Omega, C_0}\bra{\frac{1}{|\omega_1|}+\frac{1}{|\omega_2|}}\D(u),
			\end{aligned}
		\end{equation*}
		where we used Lemmas \ref{lem7} and \ref{lem7_1} at the last step.
	\end{proof}
	
	\begin{lemma}\label{lem12}
		Under the assumptions in Theorem \ref{thm:measurable}, it holds
		\begin{equation*}
			\E(u|u_\infty) \lesssim \bra{\frac{1}{|\omega_1|} + \frac{1}{|\omega_2|}} \D(u).
		\end{equation*}
	\end{lemma}
	\begin{proof}
		The proof of this lemma follows immediately from Lemmas \ref{lem7}, \ref{lem7_1}, \ref{lem10} and \ref{lem11}.
	\end{proof}
	
	\begin{proof}[Proof of Theorem \ref{thm:measurable}]
		The global existence and boundedness of a unique weak non-negative solution is given in Proposition \ref{pro1}. It can be easily checked that the solution fulfills the weak entropy-entropy dissipation relation
		\begin{equation*} 
			\E(u(t)|u_\infty) + \int_s^t\D(u(\tau))d\tau \le \E(u(\tau)|u_\infty) \quad \forall t\ge \tau \ge 0.
		\end{equation*}
		Now using the entropy-entropy dissipation inequality in Lemma \ref{lem12} and a Gronwall's inequality (as in the proof of Theorem \ref{thm:main1}), we get the exponentially fast decay of the relative entropy
		\begin{equation*} 
			\E(u(t)|u_\infty) \le \E(u_0|u_\infty)e^{-\lambda t}
		\end{equation*}
		where $\lambda^{-1}\sim |\omega_1|^{-1} + |\omega_2|^{-1}$. The convergence in $L^1(\Omega)$-norm follows directly from Lemma \ref{lem6}.
	\end{proof}
	
	\section{Degenerate diffusion and reactions - Proof of Theorem \ref{thm:main3}}\label{sec4}
	
	Due to the degeneracy of $d_3$, the global existence of solution to \eqref{special_sys} is not obtained as quickly as previously. Nevertheless, under assumption \eqref{D3'}, we can obtain the following global existence and boundedness of solutions.
	\begin{proposition}\label{pro2}
		We work under the assumptions of Theorem \ref{thm:main3}. Then for any non-negative, bounded initial data $u_0\in L^{\infty}_+(\Omega)^3$, there exists a unique global non-negative weak solution to \eqref{special_sys}, which is bounded uniformly in time, i.e.
		\begin{equation*}
			\sup_{t\ge 0}\sup_{i=1,2,3}\|u_i(t)\|_{\LO{\infty}} < +\infty.
		\end{equation*}
	\end{proposition}
	\begin{remark}
		In the proof of Proposition \ref{pro2}, we need to know that $\nabla d_3 \in \LO{q}$ for some $q>n$, in order to deal with the degeneracy \eqref{D3'}. If $d_3$ is merely continuous on $\Omega$ and its zero-set $\{x\in\Omega: d_3(x) =0\}$ is a finite union of $(n-1)$-dimensional smooth manifolds, then we can show the global existence and boundedness of solutions without imposing any conditions on  the gradient of $d_3$. The idea is that strong compactness will hold (for the third component of the solution to an approximated problem) outside of an $\eps$-neighbourhood of the considered manifolds, and that convergence a.e.  of a subsequence on the whole domain can be obtained thanks to a diagonal extraction.
		We leave the details to the interested reader.
	\end{remark}
	\begin{proof}
		We regularize the system \eqref{special_sys} as follows: for any $\eps\in(0,1)$, we define
		\begin{equation}\label{approx_diff}
			d_{\eps1}(x,t) := d_1(x,t), \quad d_{\eps2}(x,t) := d_2(x,t), \quad d_{\eps3}(x) := d_3(x) + \eps,
		\end{equation}
		\begin{equation}\label{approx_react}
			f_{\eps i}(x,t,u) := f_i(x,t,u)\bra{1+\eps\sum_{j=1}^{3}|f_j(x,t,u)|}^{-1}, \quad i=1,2,3.
		\end{equation}
		Consider now the approximating system for $u_{\eps}  = (u_{\eps1}, u_{\eps2}, u_{\eps3})$, that is
		\begin{equation}\label{approx_sys}
			\left\{
			\begin{aligned}
				&\partial_t u_{\eps i} - \na\cdot(d_{\eps i}\na u_{\eps i}) = f_{\eps i}(x,t,u_{\eps}), &&\text{ in } \Omega\times\mathbb R_+,\\
				&d_{\eps i}\na u_{\eps i}\cdot \nu = 0, &&\text{ on } \partial\Omega\times \R_+,\\
				&u_{\eps i}(x,0) = u_{i,0}, &&\text{ in } \Omega.
			\end{aligned}
			\right.
		\end{equation}
		By applying \cite[Theorem 1.1]{fitzgibbon2021reaction}, \eqref{approx_sys} has a unique global weak solution, which is also bounded uniformly in time. It is remarked, however, that this bound depends on $\eps$ and could in principle tend to $\infty$ as $\eps\to 0$. In the following, we show therefore some uniform-in-time bounds for $u_{\eps i}$, which are independent of $\eps$. In order to do that, we consider for $p\in \mathbb N$ the energy functional
		\begin{equation*}
			\mathcal{H}_p[u_{\eps}]:= \intO\bra{\frac{4}{p+1} \bra{u_{\eps1}}^{p+1} + \frac{2}{2p + 1}\bra{u_{\eps2}}^{2p+1} + \frac{1}{4p+1}\bra{u_{\eps3}}^{4p+1}}dx.
		\end{equation*}
		Differentiating $\mathcal{H}_p[u_{\eps}]$ in $t$, using the system leads to
		\begin{align*}
			\frac{d}{dt}\mathcal{H}_p[u_{\eps}](t) = &-4p\intO d_{\eps1}(u_{\eps1})^{p-1}|\na u_{\eps1}|^2dx - 4p\intO d_{\eps2}(u_{\eps2}) ^{2p-1}|\na u_{\eps2}|^2dx\\
			&-4p\intO d_{\eps3}(u_{\eps3})^{4p-1}|\na u_{\eps3}|^2dx\\
			& - 4\intO k_1(u_{\eps2}^2-u_{\eps1})(u_{\eps2}^{2p}-u_{\eps1}^p)\bra{1+{\eps}\sum_{j=1}^{3}|f_j(x,t,u_{\eps})|}^{-1}dx \\
			& - 2\intO k_2(u_{\eps3}^2 - u_{\eps2})(u_{\eps3}^{4p} - u_{\eps2}^{2p})\bra{1+{\eps}\sum_{j=1}^{3}|f_j(x,t,u_{\eps})|}^{-1}dx \le 0.
		\end{align*}
		Therefore
		\begin{equation*}
			\mathcal{H}_p[u_{\eps}](t) \le \mathcal{H}_p[u_{\eps}](0), \quad \forall t\ge 0,
		\end{equation*}
		which entails  
		\begin{equation*}
			\begin{aligned}
				\frac{4}{p+1}\|u_{\eps1}(t)\|_{\LO{p+1}}^{p+1} + \frac{2}{2p+1}\|u_{\eps2}(t)\|_{\LO{2p+1}}^{2p+1} + \frac{1}{4p+1}\|u_{\eps3}(t)\|_{\LO{4p+1}}^{4p+1}\\
				\le \frac{4}{p+1}\|u_{1,0}\|_{\LO{p+1}}^{p+1} + \frac{2}{2p+1}\|u_{2,0}\|_{\LO{2p+1}}^{2p+1} + \frac{1}{4p+1}\|u_{3,0}\|_{\LO{4p+1}}^{4p+1}.
			\end{aligned}
		\end{equation*}
		We can take the root of order $p+1$ of both sides, then let $p\to \infty$, and  obtain that
		\begin{equation}\label{bound_approx}
			\|u_{\eps1}(t)\|_{\LO{\infty}} + \|u_{\eps2}(t)\|_{\LO{\infty}}^2 + \|u_{\eps3}(t)\|_{\LO{\infty}}^4 \le C\bra{\|u_{1,0}\|_{\LO{\infty}} + \|u_{2,0}\|_{\LO{\infty}}^2 + \|u_{3,0}\|_{\LO{\infty}}^4},
		\end{equation}
		for a constant $C$ \textit{independent of $\eps\in(0,1)$}. Thanks to this and \eqref{approx_react}, we have
		\begin{equation}\label{bound_nonlinearities}
			\sup_{\eps\in(0,1)}\sup_{i=1,2,3}\|f_i(\cdot,\cdot,u_{\eps})\|_{L^{\infty}(0,T;\LO{\infty})} < +\infty.
		\end{equation}		
		Since $d_{\eps1}$ and $d_{\eps2}$ are in fact independent of $\eps$, see \eqref{approx_diff}, the classical Aubin-Lions lemma gives the strong convergence of $u_{\eps1}$ and $u_{\eps2}$, up to some subsequence, i.e.
		\begin{equation*}
			u_{\eps i} \to u_i \quad \text{strongly in } L^2(0,T;L^2(\Omega)) \text{ and weakly in } L^2(0,T;H^1(\Omega))
		\end{equation*}
		for some $u_1, u_2 \in L^2(0,T;H^1(\Omega))$. At the same time, \eqref{bound_approx} implies that
		\begin{equation}\label{g0}
			u_{\eps3} \rightharpoonup u_3 \text{ weakly in } L^2(0,T;L^2(\Omega)).
		\end{equation}
		We will now show the strong convergence of $u_{\eps3}$. By multiplying by  $u_{\eps3}$ the third equation, and by integrating  on $\Omega\times(0,T)$, we get
		\begin{equation}\label{g1}
			\int_0^T\int_{\Omega}d_{\eps3}\, |\nabla u_{\eps3}|^2dxdt \le \frac  12\|u_{3,0}\|_{\LO{2}}^2 + \int_0^T\int_{\Omega}f_{\eps3}(x,t,u_{\eps})u_{\eps3}dxdt.
		\end{equation}
		Thanks to \eqref{approx_diff}, \eqref{bound_approx} and \eqref{bound_nonlinearities}, we see that
		\begin{equation}\label{g5}
			\sup_{\eps\in(0,1)}\int_0^T\int_{\Omega}d_{\eps3}|\nabla u_{\eps3}|^2dx < +\infty.
		\end{equation}
		Thanks to this estimate, the assumption on the gradient of $d_3$, and the elementary computation
		\begin{equation*}
			|\na (d_3u_{\eps3})| \le |\na d_3|u_{\eps3} + \sqrt{d_3}\cdot\sqrt{d_3}|\na u_{\eps3}|,
		\end{equation*}
		it follows that
		\begin{equation}\label{g3}
			\sup_{\eps\in(0,1)}\|\na(d_3u_{\eps3})\|_{L^2(0,T;L^2(\Omega))} <+\infty.
		\end{equation}
		For $\varphi \in L^2(0,T;H^1(\Omega))$, we have
		\begin{equation}\label{g2}
			\begin{aligned}
				\left|\int_0^T\langle \partial_t (d_3u_{\eps3}), \varphi\rangle dt\right| &\le \left|\int_0^T\int_{\Omega}d_{\eps3}\na u_{\eps3} \cdot\na (d_3\varphi)dxdt\right| + \int_0^T\int_{\Omega}|f_3(x,t,u_{\eps})||d_3\varphi|dxdt\\
				&\le \left|\int_0^T\intO\sqrt{d_{\eps3}}\bra{\sqrt{d_{\eps3}}\na u_{\eps3}} \, \varphi \,\na d_3 \, dxdt\right|
				\\
				& +\left|\int_0^T\intO\sqrt{d_{\eps3}}\bra{\sqrt{d_{\eps3}}\na u_{\eps3}}\,d_3 \,\na \varphi \, dxdt\right|\\
				&\quad + \, \|d_3f_3(\cdot,\cdot,u_{\eps})\|_{L^\infty(0,T;\LO{\infty})}\sqrt{T|\Omega|}\|\varphi\|_{L^2(0,T;\LO{2})}.
			\end{aligned}
		\end{equation}
		The second term and the first term on the right hand side of \eqref{g2} are estimated as
		\begin{equation*}
			\left|\int_0^T\intO\sqrt{d_{\eps3}}\bra{\sqrt{d_{\eps3}}\na u_{\eps3}}d_3 \na \varphi\, dxdt\right| \le \|\sqrt{d_{\eps3}}d_3\|_{\LO{\infty}}\|\sqrt{d_{\eps3}}\na u_{\eps3}\|_{L^2(0,T;L^2(\Omega))}\|\na \varphi\|_{L^2(0,T;L^2(\Omega))}
		\end{equation*}
		and
		\begin{equation*}
			\begin{aligned}
				&\left|\int_0^T\intO\sqrt{d_{\eps3}}\bra{\sqrt{d_{\eps3}}\na u_{\eps3}}\, \varphi\, \na d_3 \, dxdt\right|\\
				&\le C \|\sqrt{d_{\eps3}}\|_{\LO{\infty}}\|\sqrt{d_{\eps3}}\na u_{\eps3}\|_{L^2(0,T;L^2(\Omega))}\|\na d_3\|_{L^{q}(\Omega)}\|\varphi\|_{L^2(0,T;L^{2^*_n}(\Omega))}
			\end{aligned}
		\end{equation*}
		thanks to $d_3\in W^{1,q}(\Omega)$, $q>n$, and $H^1(\Omega)\subset L^{2^*_n}(\Omega)$,
		where $2^*_n = +\infty$ for $n = 1$, $2^*_n < \infty$ arbitrary for $n=2$, and $2^*_n = 2n/(n-2)$ for $n\ge 3$. Thus
		\begin{equation}\label{g4}
			\{\partial_t (d_3u_{\eps3})\}_{\eps\in(0,1)} \text{ is bounded in } L^2(0,T;(H^1(\Omega))').
		\end{equation}
		From \eqref{g3} and \eqref{g4}, we can use Aubin-Lions lemma to get, up to a subsequence
		\begin{equation*}
			d_3u_{\eps3} \to \xi \quad \text{ strongly in } L^2(0,T;L^2(\Omega)).
		\end{equation*}
		Since $|\{x\in\Omega: d_3(x)\}| = 0$, we have $d_3(x)>0$ a.e. in $\Omega$. Therefore, it follows that
		\begin{equation*}
			u_{\eps3} \to \frac{\xi}{d_3} \quad \text{a.e. in } \Omega\times (0,T).
		\end{equation*}
		From this and \eqref{bound_approx}, we have $u_{\eps3} \to \dfrac{\xi}{d_3}$ strongly in $L^2(\Omega\times(0,T))$, which in combination with \eqref{g0} yields $u_3 = \dfrac{\xi}{d_3}$. Moreover, interpolating with the bounds in \eqref{bound_approx} gives
		\begin{equation*}
			u_{\eps3} \to u_3 \quad \text{strongly in } L^p(0,T;L^p(\Omega)) \quad \forall p\in [1,\infty).
		\end{equation*}
		From \eqref{g5} and boundedness of $d_3$,
		\begin{equation}\label{g6}
			d_{\eps3}\nabla u_{\eps3} \rightharpoonup \chi \quad \text{weakly in } L^2(0,T;L^2(\Omega)).
		\end{equation}
		For any smooth vector field $\psi\in C_c^\infty((0,T)\times\Omega)^n$, using $\na d_{\eps3} = \na d_3$,
		\begin{align*}
			&\int_0^T\intO(d_{\eps3}\na u_{\eps3} - \blue{[\nabla(d_3 u_3) - u_3\nabla d_3]})\cdot \psi dxdt\\
			&\blue{=\int_0^T\int_{\Omega} [(d_{\eps3}\na u_{\eps3} + u_3\na d_3)\cdot \psi + d_3u_3\na\cdot \psi]\ dxdt}\\
			&= - \int_0^T\int_{\Omega} (u_{\eps3}-u_3)(\na d_3 \cdot \psi + d_{\eps3}\na\cdot\psi)dxdt -\int_0^T\intO u_3 (d_{\eps3} - d_3)\na \cdot \psi dxdt\\
			&\longrightarrow 0 \quad \text{ as $\eps \to 0$}.
		\end{align*}
		
		\medskip
		
		This means that $d_{\eps3}\na u_{\eps3}$ converges to $\blue{\nabla(d_3 u_3) - u_3\nabla d_3}$ in the sense of distributions. Together with \eqref{g6}, we finally obtain 
		\begin{equation*}
			d_{\eps3} \na u_{\eps3} \rightharpoonup \blue{\nabla(d_3 u_3) - u_3\nabla d_3} \quad \text{weakly in } L^2(0,T;L^2(\Omega)).
		\end{equation*}
		Now we can pass to the limit in the weak formulation of the approximating system
		\begin{equation*}
			\int_0^T\langle \partial_t u_{\eps i}, \varphi\rangle dt + \int_0^T\int_{\Omega}d_{\eps i}\na u_{\eps i}\cdot \na \varphi dxdt = \int_0^T\intO f_{\eps i}(u_{\eps})\varphi dxdt, \qquad \varphi\in L^2(0,T;H^1(\Omega)) ,
		\end{equation*}
		to conclude that $u = (u_1,u_2,u_3)$ is a weak solution to \eqref{special_sys} on $(0,T)$ for $T>0$ arbitrary. Moreover, this solution is bounded uniformly in time thanks to \eqref{bound_approx}, and consequently is unique due to the local Lipschitz continuity of the nonlinearities.
	\end{proof}
	
	\medskip
	To show the convergence to equilibrium, we use  some estimates which are similar to those used in the proof of Theorem \ref{thm:measurable}. We present them below:
	\begin{lemma}\label{new_lem}
		Under the assumptions in Theorem \ref{thm:main3}, we have the following estimates for solutions to \eqref{special_sys}
		\begin{align*}
			&\sum_{j=1}^3\|u_j-u_{j,\infty}\|_{L^1(\Omega)}^2 \lesssim_{u_\infty,C_0} \E(u|u_\infty) \lesssim_{u_\infty} \sum_{j=1}^3\int_{\Omega}|u_j - u_{j,\infty}|^2dx,\\
			&\norm{\sqrt{u_j} - \avO{\sqrt{u_j}}}_{L^2(\Omega)}^2 \lesssim_{\Omega}\intO d_j\frac{|\na u_j|^2}{u_j}dx, \quad j=1,2,\\
			&\norm{u_j - \avO{u_j}}_{L^2(\Omega)}^2  \lesssim_{\Omega,C_0} \intO d_j\frac{|\na u_j|^2}{u_j}dx, \quad j=1,2,\\
			&\D(u) \gtrsim  \sum_{j=1}^3\int_{\Omega}d_j\frac{|\na u_j|^2}{u_j}dx + \int_{\omega_1}\abs{u_2 - \sqrt{u_1}}^2dx + \int_{\omega_2}\abs{u_3 - \sqrt{u_2}}^2dx,\\
			&|u_1 - u_{1,\infty}|^2  \lesssim_{C_0,u_\infty} |u_2 - \sqrt{u_1}|^2 + |u_2 - u_{2,\infty}|^2,\\
			&|u_3 - u_{3,\infty}|^2  \lesssim_{u_\infty} |u_3 - \sqrt{u_2}|^2 + |u_2 - u_{2,\infty}|^2,\\
			&|u_2 - u_{2,\infty}|^2 \lesssim_{C_0} \sum_{j=1}^3|u_j - \avO{u_j}|^2   + |u_2 - \sqrt{u_1}|^2 + |u_3 - \sqrt{u_2}|^2 ,
		\end{align*}
		where
		\begin{equation*} 
			C_0:= \sup_{t\ge 0}\sum_{i=1}^3\|u_i(t)\|_{\LO{\infty}}.
		\end{equation*}
	\end{lemma}
	\begin{proof}
		The proofs of these estimates are the same as the proofs of Lemmas \ref{lem6}, \ref{lem7}, \ref{lem7_1}, \ref{lem8}, since we do not need a positive lower bound for $d_3$ in these proofs.
	\end{proof}

	\medskip
	Comparing to the proof of Theorem \ref{thm:measurable}, we need some estimates to compensate the lack of diffusion of $\S_3$ in some part of $\Omega$. This is done in the following lemma.
	\begin{lemma}\label{lem13}
		Under the assumptions of Theorem \ref{thm:main3}, it holds for solutions to \eqref{special_sys},
		\begin{equation*} 
			\int_{\Omega}\abs{u_3 - \avO{u_3}}^2dx \lesssim \D(u).
		\end{equation*}
	\end{lemma}
	\begin{proof}
		Because $\omega _{2}$ is open with Lipschitz boundary, 
		the Poincar\'{e}-Wirtinger inequality holds for $\omega _{2}$. The assumption $%
		d_{3}\in W^{1,q}\left( \Omega \right) $ implies $d_{3}\in C(\overline{\Omega})$ thanks to Sobolev embedding. From the assumption $\{ x\in 
		\overline{\Omega }:d_{3}\left( x\right) =0\} \subset \omega
		_{2}$, there
		is $B$ an open set of class $C^{1}$ such that $\{ x\in \overline{\Omega}:d_{3}\left( x\right) =0\} \subset B\Subset \omega _{2}$ with the
		properties:
		\begin{itemize}
			\item the Poincar\'{e}-Wirtinger inequality in $B^{c}=\Omega \left\backslash
			B\right. $ holds;
			
			\item $d_{3}\left( x\right) \geq \varrho,\, \forall x\in B^{c}$ for some $\varrho>0$;
			
			\item $\omega _{2}^{c}\subset B^{c};$
			
			\item $\omega _{2}\left\backslash B\right. \subset \omega _{2}$ and $\omega
			_{2}\left\backslash B\right. \subset B^{c}.$
		\end{itemize}
		Next, using triangle inequalities, we estimate%
		\begin{eqnarray}
			\int_{\Omega }\left\vert u_{3}-\left[ u_{3}\right] _{\Omega }\right\vert
			^{2}dx &\lesssim &\int_{\omega _{2}^{c}}\left\vert u_{3}-\left[ u_{3}\right]
			_{\Omega }\right\vert ^{2}dx+\int_{\omega _{2}}\left\vert u_{3}-\left[ u_{3}%
			\right] _{\Omega }\right\vert ^{2}dx  \notag \\
			&\lesssim &\int_{\omega _{2}^{c}}\left\vert u_{3}-\left[ u_{3}\right]
			_{B^{c}}\right\vert ^{2}dx+\left\vert \left[ u_{3}\right] _{B^{c}}-\left[
			u_{3}\right] _{\Omega }\right\vert ^{2}  \notag \\
			&&\quad +\int_{\omega _{2}}\left\vert u_{3}-\sqrt{u_{2}}\right\vert
			^{2}dx+\int_{\omega _{2}}\left\vert \sqrt{u_{2}}-\left[ \sqrt{u_{2}}\right]
			_{\omega _{2}}\right\vert ^{2}dx  \notag \\
			&&\quad +\left\vert \left[ \sqrt{u_{2}}\right] _{\omega _{2}}-\left[ u_{3}%
			\right] _{\omega _{2}}\right\vert ^{2}+\left\vert \left[ u_{3}\right]
			_{\omega _{2}}-\left[ u_{3}\right] _{\Omega }\right\vert ^{2}.  \label{e0}
		\end{eqnarray}%
		We treat $\left\vert \left[ u_{3}\right] _{B^{c}}-\left[ u_{3}\right]
		_{\Omega }\right\vert ^{2}+\left\vert \left[ u_{3}\right] _{\omega _{2}}-%
		\left[ u_{3}\right] _{\Omega }\right\vert ^{2}$ with the aim to drop the term $%
		\left[ u_{3}\right] _{\Omega }$. Since $\left\vert \omega _{2}\right\vert
		+\left\vert \omega _{2}^{c}\right\vert =\left\vert \Omega \right\vert =1$
		and $\left[ u_{3}\right] _{\Omega }=\left\vert \omega _{2}\right\vert \left[
		u_{3}\right] _{\omega _{2}}+\left\vert \omega _{2}^{c}\right\vert \left[
		u_{3}\right] _{\omega _{2}^{c}}$, it holds 
		\begin{equation*}
			\left\vert \left[ u_{3}\right] _{\omega _{2}}-\left[ u_{3}\right] _{\Omega
			}\right\vert ^{2}=\left\vert \omega _{2}^{c}\right\vert ^{2}\left\vert \left[
			u_{3}\right] _{\omega _{2}}-\left[ u_{3}\right] _{\omega
				_{2}^{c}}\right\vert ^{2}.
		\end{equation*}%
		Therefore, by triangle inequality, we have%
		\begin{equation}\label{k1}
			\begin{aligned}
				\left\vert \left[ u_{3}\right] _{B^{c}}-\left[ u_{3}\right] _{\Omega
				}\right\vert ^{2}+\left\vert \left[ u_{3}\right] _{\omega _{2}}-\left[ u_{3}%
				\right] _{\Omega }\right\vert ^{2} &\lesssim \left\vert \left[ u_{3}\right]
				_{B^{c}}-\left[ u_{3}\right] _{\omega _{2}}\right\vert ^{2}+\left\vert \left[
				u_{3}\right] _{\omega _{2}}-\left[ u_{3}\right] _{\Omega }\right\vert ^{2} \\
				&\lesssim \left\vert \left[ u_{3}\right] _{B^{c}}-\left[ u_{3}\right]
				_{\omega _{2}}\right\vert ^{2}+\left\vert \left[ u_{3}\right] _{\omega _{2}}-%
				\left[ u_{3}\right] _{\omega _{2}^{c}}\right\vert ^{2} \\
				&\lesssim \left\vert \left[ u_{3}\right] _{B^{c}}-\left[ u_{3}\right]
				_{\omega _{2}}\right\vert ^{2}+\left\vert \left[ u_{3}\right] _{B^{c}}-\left[
				u_{3}\right] _{\omega _{2}^{c}}\right\vert ^{2}.
			\end{aligned}
		\end{equation}%
		We bound $\left\vert \left[ u_{3}\right] _{B^{c}}-\left[ u_{3}\right]
		_{\omega _{2}^{c}}\right\vert ^{2}$ thanks to Cauchy-Schwarz inequality as
		follows%
		\begin{equation}\label{k2}
			\left\vert \left[ u_{3}\right] _{B^{c}}-\left[ u_{3}\right] _{\omega
				_{2}^{c}}\right\vert ^{2}=\left\vert \frac{1}{\left\vert \omega
				_{2}^{c}\right\vert }\int_{\omega _{2}^{c}}\left( \left[ u_{3}\right]
			_{B^{c}}-u_{3}\right) dx\right\vert ^{2}\leq \frac{1}{\left\vert \omega
				_{2}^{c}\right\vert }\int_{\omega _{2}^{c}}\left\vert \left[ u_{3}\right]
			_{B^{c}}-u_{3}\right\vert ^{2}dx.
		\end{equation}%
		Now, we estimate $\left\vert \left[ u_{3}\right] _{B^{c}}-\left[ u_{3}\right]
		_{\omega _{2}}\right\vert ^{2}$. Let $\vartheta :=\omega _{2}\left\backslash
		B\right. $. Since $\vartheta \subset \omega _{2}$ and $\vartheta \subset
		B^{c}$, it holds%
		\begin{eqnarray}
			\quad \left\vert \left[ u_{3}\right] _{B^{c}}-\left[ u_{3}\right] _{\omega
				_{2}}\right\vert ^{2} &=&\frac{1}{\left\vert \vartheta \right\vert }%
			\int_{\vartheta }\left\vert \left[ u_{3}\right] _{B^{c}}-u_{3}+u_{3}-\left[
			u_{3}\right] _{\omega _{2}}\right\vert ^{2}dx\label{k3} \\
			&\lesssim &\int_{\vartheta }\left\vert u_{3}-\left[ u_{3}\right]
			_{B^{c}}\right\vert ^{2}dx\nonumber \\
			&&\quad +\int_{\vartheta }\left\vert u_{3}-\sqrt{u_{2}}\right\vert
			^{2}dx+\int_{\vartheta }\left\vert \sqrt{u_{2}}-\left[ \sqrt{u_{2}}\right]
			_{\omega _{2}}\right\vert ^{2}dx+\left\vert \left[ \sqrt{u_{2}}\right]
			_{\omega _{2}}-\left[ u_{3}\right] _{\omega _{2}}\right\vert ^{2}\nonumber \\
			&\lesssim &\int_{B^{c}}\left\vert u_{3}-\left[ u_{3}\right]
			_{B^{c}}\right\vert ^{2}dx\nonumber \\
			&&\quad +\int_{\omega _{2}}\left\vert u_{3}-\sqrt{u_{2}}\right\vert
			^{2}dx+\int_{\omega _{2}}\left\vert \sqrt{u_{2}}-\left[ \sqrt{u_{2}}\right]
			_{\omega _{2}}\right\vert ^{2}dx+\left\vert \left[ \sqrt{u_{2}}\right]
			_{\omega _{2}}-\left[ u_{3}\right] _{\omega _{2}}\right\vert^{2}\nonumber.
		\end{eqnarray}%
		Combining the estimates \eqref{k1}, \eqref{k2} and \eqref{k3} with \eqref{e0}, one can conclude that%
		\begin{eqnarray*}
			\int_{\Omega }\left\vert u_{3}-\left[ u_{3}\right] _{\Omega }\right\vert
			^{2}dx &\lesssim &\int_{\omega _{2}^{c}}\left\vert u_{3}-\left[ u_{3}\right]
			_{B^{c}}\right\vert ^{2}dx+\int_{B^{c}}\left\vert u_{3}-\left[ u_{3}\right]
			_{B^{c}}\right\vert ^{2}dx \\
			&&\quad +\int_{\omega _{2}}\left\vert u_{3}-\sqrt{u_{2}}\right\vert
			^{2}dx+\int_{\omega _{2}}\left\vert \sqrt{u_{2}}-\left[ \sqrt{u_{2}}\right]
			_{\omega _{2}}\right\vert ^{2}dx \\
			&&\quad +\left\vert \left[ \sqrt{u_{2}}\right] _{\omega _{2}}-\left[ u_{3}%
			\right] _{\omega _{2}}\right\vert ^{2} \\
			&&=:\left( I\right) +\left( II\right) +\left( III\right) +\left( IV\right)
			+\left( V\right) .
		\end{eqnarray*}%
		Due to $\omega _{2}^{c}\subset B^{c}$, Poincar\'{e}-Wirtinger inequality on $%
           B^{c}$ and $d_{3}\geq \varrho>0$ on $B^{c}$, 
		\begin{equation*}
			\left( I\right) +\left( II\right) \lesssim \int_{B^{c}}\left\vert u_{3}- 
			\left[ u_{3}\right] _{B^{c}}\right\vert ^{2}dx\lesssim
			\int_{B^{c}}\left\vert \nabla u_{3}\right\vert ^{2}dx\lesssim \int_{\Omega
			}d_{3}\frac{\left\vert \nabla u_{3}\right\vert ^{2}}{u_{3}}dx\leq 
			\mathcal{D}\left( u\right) .
        	\end{equation*}%
		By Poincar\'{e}-Wirtinger inequality on $\omega _{2}$, we have 
		\begin{equation*}
			\left( IV\right) \lesssim \int_{\omega _{2}}\frac{\left\vert \nabla
				u_{2}\right\vert ^{2}}{u_{2}}dx\lesssim \int_{\Omega }d_{2}\frac{\left\vert
				\nabla u_{2}\right\vert ^{2}}{u_{2}}dx\leq \mathcal{D}\left( u\right) .
		\end{equation*}%
		Besides, thanks to Cauchy-Schwarz inequality and Lemma \ref{new_lem},
		\begin{equation*}
			\left( III\right) +\left( V\right) \lesssim \int_{\omega _{2}}\left\vert u_{3}-\sqrt{u_{2}}%
			\right\vert ^{2}\lesssim \mathcal{D}\left( u\right) .
		\end{equation*}%
		This concludes the proof of Lemma \ref{lem13}.
	\end{proof}
	
	We are now ready to prove Theorem \ref{thm:main3}.
	\begin{proof}[Proof of Theorem \ref{thm:main3}]
		Thanks to Lemma \ref{lem13}, we can use the same arguments in the proof of Theorem \ref{thm:measurable} to obtain the entropy-entropy dissipation inequality
		\begin{equation*} 
			\mathcal{E}(u|u_\infty)\lesssim \mathcal{D}(u).
		\end{equation*}
		Indeed, the same arguments as in Lemma \ref{lem10} gives
		\begin{equation*} 
			\E(u|u_\infty) \lesssim  \sum_{j=1}^3\|u_j- \avO{u_{j}}\|_{\LO{2}}^2 + \int_{\omega_1}\bra{|u_2 - \sqrt{u_1}|^2 + |u_3 - \sqrt{u_2}|^2}dx.
		\end{equation*}
		It follows from Lemma \ref{new_lem} that
		\begin{equation*} 
			\sum_{j=1}^2\|u_j- \avO{u_{j}}\|_{\LO{2}}^2 + \int_{\omega_1}|u_2 - \sqrt{u_1}|^2dx \lesssim \D(u).
		\end{equation*}
		Lemma \ref{lem13} gives
		\begin{equation*} 
			\|u_3 - \avO{u_3}\|_{\LO{2}}^2 \lesssim \D(u).
		\end{equation*}
		With this estimate at hand, we can finally estimate
		\begin{equation*} 
			\int_{\omega_1}|u_3 - \sqrt{u_2}|^2dx \lesssim \D(u)
		\end{equation*}
		by using the same arguments as in Lemma \ref{lem11}.

		\medskip
		
		The convergence to equilibrium then follows in the standard way, as in Proof of Theorem \ref{thm:measurable}, so we omit the details.
	\end{proof}
	\medskip
	\noindent{\bf Acknowledgement.} B.Q. Tang is supported by NAWI Graz and received funding from the FWF project ``Quasi-steady-state approximation for PDE'', number I-5213. 
	
	This work was partly carried out during the visits of the first author to the University of Graz, and of the last author to the Universit\'e Paris Cit\'e  and the University of Orl\'eans. The universities' hospitality is greatly acknowledged.


\begin{thebibliography}{HHMM18}
		
		\bibitem[And11]{anderson2011proof}
		David~F Anderson.
		\newblock A proof of the global attractor conjecture in the single linkage
		class case.
		\newblock {\em SIAM Journal on Applied Mathematics}, 71(4):1487--1508, 2011.

				
		\bibitem[BET22]{braukhoff2022quantitative}
		Marcel Braukhoff, Amit Einav, and Bao~Quoc Tang.
		\newblock Quantitative dynamics of irreversible enzyme reaction--diffusion
		systems.
		\newblock {\em Nonlinearity}, 35(4):1876, 2022.
		
		
		\bibitem[DF06]{Desvillettes2006}
		Laurent Desvillettes and Klemens Fellner.
		\newblock Exponential decay toward equilibrium via entropy methods for
		reaction--diffusion equations.
		\newblock {\em Journal of Mathematical Analysis and Applications}, 319, 2006.
		
		\bibitem[DF07]{desvillettes2007entropy}
		Laurent Desvillettes and Klemens Fellner.
		\newblock Entropy methods for reaction-diffusion equations with degenerate
		diffusion arising in reversible chemistry.
		\newblock In {\em Proceedings of EQUADIFF}, 2007.
		
		
		\bibitem[DF14]{desvillettes2014exponential}
		Laurent Desvillettes and Klemens Fellner.
		\newblock Exponential convergence to equilibrium for nonlinear
		reaction-diffusion systems arising in reversible chemistry.
		\newblock In {\em System Modeling and Optimization: 26th IFIP TC 7 Conference,
			CSMO 2013, Klagenfurt, Austria, September 9-13, 2013, Revised Selected Papers
			26}, pages 96--104. Springer, 2014.
		
		\bibitem[DF15]{desvillettes2015duality}
		Laurent Desvillettes and Klemens Fellner.
		\newblock Duality and entropy methods for reversible reaction-diffusion
		equations with degenerate diffusion.
		\newblock {\em Mathematical Methods in the Applied Sciences},
		38(16):3432--3443, 2015.
		
		\bibitem[DFPV07]{desvillettes2007global}
		Laurent Desvillettes, Klemens Fellner, Michel Pierre, and Julien Vovelle.
		\newblock Global existence for quadratic systems of reaction-diffusion.
		\newblock {\em Advanced Nonlinear Studies}, 7(3):491--511, 2007.
		
		\bibitem[DFT17]{desvillettes2017trend}
		Laurent Desvillettes, Klemens Fellner, and Bao~Quoc Tang.
		\newblock Trend to equilibrium for reaction-diffusion systems arising from
		complex balanced chemical reaction networks.
		\newblock {\em SIAM Journal on Mathematical Analysis}, 49(4):2666--2709, 2017.
		
		
		
		
		
		\bibitem[DP22]{desvillettes2021exponential}
		Laurent Desvillettes and Kim~Dang Phung.
		\newblock Exponential decay toward equilibrium via log convexity for a
		degenerate reaction-diffusion system.
		\newblock {\em Journal of Differential Equations}, 338:227--255, 2022.
		
		\bibitem[DV00]{CV}Laurent Desvillettes and C\'edric  Villani.
		\newblock On the spatially homogeneous Landau equation for hard potentials part II: {H}-theorem and applications.
		\newblock {\em Communications in Partial Differential Equations}, Volume 25, 2000 - Issue 1-2.
		
		\bibitem[EMT20]{einav2020indirect}
		Amit Einav, Jeffrey~J Morgan, and Bao~Quoc Tang.
		\newblock Indirect diffusion effect in degenerate reaction-diffusion systems.
		\newblock {\em SIAM Journal on Mathematical Analysis}, 52(5):4314--4361, 2020.
		
		\bibitem[Fei19]{feinberg2019foundations}
		Martin Feinberg.
		\newblock Foundations of chemical reaction network theory.
		\newblock Springer 2019.
		
		
		\bibitem[Fis15]{fischer2015global}
		Julian Fischer.
		\newblock Global existence of renormalized solutions to entropy-dissipating
		reaction--diffusion systems.
		\newblock {\em Archive for Rational Mechanics and Analysis}, 218(1):553--587,
		2015.
		
		\bibitem[Fis17]{fischer2017weak}
		Julian Fischer.
		\newblock Weak--strong uniqueness of solutions to entropy-dissipating
		reaction--diffusion equations.
		\newblock {\em Nonlinear Analysis}, 159:181--207, 2017.
		
		\bibitem[FL16]{Fellner2016a}
		Klemens Fellner and El-Haj Laamri.
		\newblock Exponential decay towards equilibrium and global classical solutions
		for nonlinear reaction-diffusion systems.
		\newblock {\em Journal of Evolution Equations}, 16, 2016.
		
		\bibitem[FLT18]{fellner2018well}
		Klemens Fellner, Evangelos Latos, and Bao~Quoc Tang.
		\newblock Well-posedness and exponential equilibration of a volume-surface
		reaction--diffusion system with nonlinear boundary coupling.
		\newblock {\em Annales de l'Institut Henri Poincar{\'e} (C), Analyse non
			lin{\'e}aire}, 35(3):643--673, 2018.
		
		\bibitem[FMTY21]{fitzgibbon2021reaction}
		William~E Fitzgibbon, Jeffrey~J Morgan, Bao~Quoc Tang, and Hong-Ming Yin.
		\newblock Reaction-diffusion-advection systems with discontinuous diffusion and
		mass control.
		\newblock {\em SIAM Journal on Mathematical Analysis}, 53(6):6771--6803, 2021.
		
		
		
		\bibitem[FT18]{fellner2018convergence}
		Klemens Fellner and Bao~Quoc Tang.
		\newblock Convergence to equilibrium of renormalised solutions to nonlinear
		chemical reaction--diffusion systems.
		\newblock {\em Zeitschrift f{\"u}r angewandte Mathematik und Physik},
		69(3):1--30, 2018.
		
		\bibitem[GS22]{GS22}
		Thierry Gallay and Sini\v{s}a Slijep\v{c}evi\'{c}. 
		\newblock Diffusive relaxation to equilibria for an extended reaction–diffusion system on the real line.
		\newblock \textit{Journal of Evolution Equations} 22.2 (2022): 47.
		
		\bibitem[GGH96]{glitzky1996free}
		Annegret Glitzky, Konrad Gr\"oger, and Rolf H{\"u}nlich.
		\newblock Free energy and dissipation rate for reaction diffusion processes of
		electrically charged species.
		\newblock {\em Applicable Analysis}, 60(3-4):201--217, 1996.
		
		\bibitem[GH97]{glitzky1997energetic}
		Annegret Glitzky and Rolf H{\"u}nlich.
		\newblock Energetic estimates and asymptotics for electro-reaction-diffusion
		systems.
		\newblock {\em ZAMM-Journal of Applied Mathematics and Mechanics/Zeitschrift
			f{\"u}r Angewandte Mathematik und Mechanik}, 77(11):823--832, 1997.
		
		\bibitem[Gr{\"o}83]{groger1983asymptotic}
		Konrad Gr{\"o}ger.
		\newblock Asymptotic behavior of solutions to a class of diffusion-reaction
		equations.
		\newblock {\em Mathematische Nachrichten}, 112(1):19--33, 1983.
		
		\bibitem[HHMM18]{Haskovec2018}
		Jan Haskovec, Sabine Hittmeir, Peter Markowich, and Alexander Mielke.
		\newblock Decay to equilibrium for energy-reaction-diffusion systems.
		\newblock {\em SIAM Journal on Mathematical Analysis}, 50, 2018.
		
		\bibitem[MHM15]{Mielke2015}
		Alexander Mielke, Jan Haskovec, and Peter Markowich.
		\newblock On uniform decay of the entropy for reaction--diffusion systems.
		\newblock {\em Journal of Dynamics and Differential Equations}, 27, 2015.
		
		\bibitem[MM18]{mielke2018convergence}
		Alexander Mielke and Markus Mittnenzweig.
		\newblock Convergence to equilibrium in energy-reaction--diffusion systems
		using vector-valued functional inequalities.
		\newblock {\em Journal of Nonlinear Science}, 28(2):765--806, 2018.

		\bibitem[MS24]{mielke2024convergence}
		Alexander Mielke and Stefanie Schindler. 
		\newblock Convergence to self-similar profiles in reaction-diffusion systems.
		\newblock {\em SIAM Journal on Mathematical Analysis} 56.6 (2024): 7108-7135.
		
		\bibitem[Pie10]{pierre2010global}
		Michel Pierre.
		\newblock Global existence in reaction-diffusion systems with control of mass:
		a survey.
		\newblock {\em Milan Journal of Mathematics}, 78:417--455, 2010.
		
		\bibitem[PSZ17]{Pierre2017}
		Michel Pierre, Takashi Suzuki, and Rong Zou.
		\newblock Asymptotic behavior of solutions to chemical reaction--diffusion
		systems.
		\newblock {\em Journal of Mathematical Analysis and Applications}, 450, 2017.
		
		
	\end{thebibliography}
\end{document}